\documentclass[11pt]{amsart}

\usepackage[latin1]{inputenc}
\usepackage{color,mathrsfs,amssymb,hyperref} 

\input xy
\xyoption{all}
\usepackage{ifpdf}
\ifpdf 
  \usepackage[pdftex]{graphicx}
  \DeclareGraphicsExtensions{.pdf,.png,.jpg,.jpeg,.mps}
  \usepackage{pgf}
  \usepackage{tikz}
\else 
  \usepackage{graphicx}
  \DeclareGraphicsExtensions{.eps,.bmp}
  \usepackage{pgf}
  \usepackage{tikz}
  \usepackage{pstricks}
\fi

\newcommand{\C}{{\mathbb C}}

\newcommand{\N}{{\mathbb N}}

\newcommand{\R}{{\mathbb R}}

\newcommand{\Z}{{\mathbb Z}}

\newcommand{\fa}{{\mathfrak a}}

\newcommand{\fd}{{\mathfrak d}}

\newcommand{\fm}{{\mathfrak m}}

\newcommand{\cH}{{\mathcal H}}

\newcommand{\cL}{{\mathcal L}}
\newcommand{\cM}{{\mathcal M}}

\newcommand{\scrA}{\mathscr{A}}
\newcommand{\scrB}{\mathscr{B}}
\newcommand{\scrC}{\mathscr{C}}

\newcommand{\scrS}{\mathscr{S}}

\newcommand{\ga}{\alpha}
\newcommand{\gb}{\beta}
\renewcommand{\gg}{\gamma}
\newcommand{\gG}{\Gamma}
\newcommand{\gd}{\delta}

\newcommand{\gve}{\varepsilon}
\newcommand{\gk}{\kappa}
\newcommand{\gl}{\lambda}

\newcommand{\go}{\omega}

\newcommand{\gvp}{\varphi}
\newcommand{\gt}{\theta}

\DeclareMathOperator{\ctg}{ctg}
\DeclareMathOperator{\tg}{tg}

\DeclareMathOperator{\diag}{diag}

\DeclareMathOperator{\im}{Im}

\DeclareMathOperator\Mp{Mp}
\DeclareMathOperator{\op}{op} 

\DeclareMathOperator{\ord}{ord}
\DeclareMathOperator\Sp{Sp}

\DeclareMathOperator{\re}{Re}
\DeclareMathOperator{\res}{res}
\DeclareMathOperator{\sgn}{sgn}

\DeclareMathOperator{\Tr}{Tr} 
\DeclareMathOperator{\tr}{tr}

\DeclareMathOperator{\U}{U}

\newtheorem{theorem}{Theorem}[section]
\newtheorem{corollary}[theorem]{Corollary}
\newtheorem{lemma}[theorem]{Lemma}
\newtheorem{proposition}[theorem]{Proposition}

\theoremstyle{definition}
\newtheorem{remark}[theorem]{Remark}
\newtheorem{definition}[theorem]{Definition}

\def\skp#1{\langle#1\rangle}
\def\mat#1{\begin{pmatrix}#1\end{pmatrix}}
\def\Res{\mathrm{Res}_{z=0}}

\renewcommand{\labelenumi}{(\roman{enumi})}

\title[Noncommutative Residues, Equivariant Traces, Trace Expansions] {Noncommutative Residues, Equivariant Traces, and Trace Expansions for an Operator Algebra on $\R^n$} 
\author{Anton Savin}
\address{A. Savin. Peoples' Friendship University of Russia (RUDN University),
6 Miklukho-Maklaya St, Moscow, 117198, Russian Federation}
\email{a.yu.savin@gmail.com} 

\author{Elmar Schrohe} 
\address{E. Schrohe. Leibniz University Hannover, Institute of Analysis, Welfengarten 1, 30167 Hannover, Germany }
\email{schrohe@math.uni-hannover.de} 
\date{}

\begin{document}
\maketitle
\begin{abstract} 
We consider an algebra $\mathscr A$ of Fourier integral operators on $\mathbb R^n$. It consists of all operators $D: \mathscr S(\mathbb R^n)\to \mathscr S(\mathbb R^n)$ on the Schwartz space $\mathscr S(\mathbb R^n)$ 
that can be written as finite sums 
\begin{eqnarray}
\label{algebra}
D= \sum R_gT_w A,
\end{eqnarray} 
with  Shubin type pseudodifferential operators $A$, Heisenberg-Weyl operators $T_w$, $w\in \mathbb C^n$, and lifts $R_g$, $g\in \U(n)$, of unitary matrices $g$ on $\mathbb C^n$ to operators $R_g$ in the complex metaplectic group. 

For $D \in \mathscr A$ and a suitable auxiliary Shubin pseudodifferential operator $H$ we establish expansions for $\mathop{\mathrm {Tr}}(D(H-\lambda)^{-K})$ as $|\lambda| \to \infty$ in a sector of 
$\mathbb C$ for sufficiently large $K$ and of $\mathop{\mathrm {Tr}}(De^{-tH})$ as $t\to 0^+$. We also obtain the singularity structure of the meromorphic extension of 
$z\mapsto \mathop{\mathrm{Tr}}(DH^{-z})$ to $\mathbb C$. 
 
Moreover, we find a  noncommutative residue  as a suitable coefficient in these  expansions and construct from it a family  of localized equivariant traces on the algebra.

\end{abstract}

\tableofcontents 

\section{Introduction}  
Traces play an important role in many areas of mathematics and physics. This holds particularly true for noncommutative geometry, where a remarkable observation was made by Wodzicki  \cite{Wod1,Wod2} in 1984: He showed that, on the algebra of all classical pseudodifferential operators on a closed manifold $M$ of dimension $\ge2$, there exists - up to multiples -  only one trace, the so-called noncommutative residue, denoted `$\res$'. Given a pseudodifferential operator $A$ of order $\fm\in \Z$ whose symbol $a$ has an asymptotic expansion $a\sim\sum_{j\ge0} a_{\fm-j}$ into terms $a_{\fm-j}(x,\xi)$ that are positively homogeneous of degree $\fm-j$ in $\xi$, 
$$\res A = (2\pi)^{-\dim M} \int_{S^*M} a_{-\dim M} (x,\xi) dS.$$

Connes' result \cite{Connes88} that the noncommutative residue coincides with the Dixmier trace on pseudodifferential operators of order $-\dim M$ then allowed Connes and Moscovici 
\cite{CoMo1} to write down index formulae in noncommutative geometry that are explicitly computable in many cases. It was shown by Ponge \cite{Pon4} that their formula indeed gives the classical local index formula for the Dirac spectral triple.

The noncommutative residue was discovered independently by Guillemin in his `soft' proof of Weyl's formula on the asymptotic distribution of eigenvalues \cite{Gu}. He introduced  $\res D$  actually in a more abstract setting as the residue in $z=0$ of the meromorphic extension of the operator trace $\Tr(DH^{-z})$ for an auxiliary pseudodifferential operator $H$ which he assumed to be positive and of order $1$. The residue then turned out to be independent of the choice of $H$.

The study of expressions like $\Tr(DH^{-z})$ for an arbitrary pseudodifferential operator $D$ and a fixed auxiliary operator $H$
generalized  Seeley's work on complex powers of elliptic pseudodifferential operators \cite{Seeley67}; it  was taken up systematically by Grubb and Seeley. 
In \cite{GS1}, they developed their {\em  weakly parametric calculus} adapted to the analysis of expressions $\Tr(D(H-\lambda)^{-K})$ for an arbitrary operator $D$ and the $K$-th power of the resolvent of the auxiliary operator $H$. They obtained an expansion of this trace into powers of $\lambda$ and logarithmic terms  as $\gl \to \infty$ in a sector of the complex plane.  
Moreover, as they showed in \cite{GS2}, similar expansions for $\Tr(De^{-tH})$ as $t\to 0^+$ and information on the structure of the singularities of the meromorphic extension of $\Tr(DH^{-z})$ to the complex plane can be derived from it. In fact, under rather mild conditions, the analysis of all three of these expansions is equivalent. 

Analogs of Wodzicki's residue and connections to trace expansions meanwhile have been established for many situations;  e.g.  for anisotropic \cite{BoNi}, log-polyhomogeneous \cite{Lesch99}, SG (or scattering) type \cite{Nic03}, bisingular \cite{NR}, conically degenerate \cite{Sch97}, \cite{GL02},  and projective \cite{SeSt13}   pseudodifferential operators, as well as for boundary value problems \cite{FGLS}, \cite{GSc1}, certain Fourier integral operators \cite{Gu93}, foliations \cite{BN}, 
Heisenberg manifolds \cite{Ponge07}, noncommutative tori \cite{FaWo11}, \cite{LNP2016}   and for algebras involving group \cite{D1} or groupoid actions   \cite{Perrot14}, to mention just a few. 

In the present article, we consider an algebra $\scrA$ of Shubin type Fourier integral operators on $\R^n$. 
Algebraically, $\scrA$ is the (algebraic)  twisted  crossed product 
$$\scrA= \Psi\rtimes (\C^n\rtimes \U(n))$$
of the algebra $\Psi$ of all classical  Shubin pseudodifferential operators on $\R^n$ with the group 
$\C^n\rtimes \U(n)$, where $\U(n)$ denotes the unitary operators on $\C^n$, which we identify with $T^*\R^n=\R^n\times \R^n$ via $(x,p)\mapsto z=p+ix$. 
The algebra $\scrA$ has a natural representation on the Schwartz space $\scrS(\R^n)$ of rapidly decreasing  functions on $\R^n$. In this case,
$\C^n$ is represented by Heisenberg-Weyl operators: For $w=a-ik\in \C^n$, the operator $T_w$ is defined by 
$$T_w u(x)  = e^{ikx-iak/2} u(x-a), \quad u\in \scrS(\R^n).$$
Moreover, $\U(n)$ is represented by operators $R_g$, $g\in \U(n)$,  in the complex metaplectic group $\Mp^c(n)$. Full details on all operator classes will be given, below. For treatises on the metaplectic group we refer to Leray \cite{Ler8} and de Gosson \cite{deG1}, see also
\cite{SaSch2}.

Summing up, we consider operators $D$ which are finite sums
 \begin{eqnarray}
\label{eq_D}
D= \sum R_{g} T_{w} A,
\end{eqnarray}
with classical Shubin pseudodifferential operators $A$, Heisenberg-Weyl operators $T_w$ and 
metaplectic operators $R_g$. 

This algebra has been introduced in \cite{SaSch3}, where explicit formulae for the Chern-Connes character in the local index formula of Connes and Moscovici \cite{CoMo1} were derived and applications to noncommutative tori and toric orbifolds were given. 
For the subalgebra without the Heisenberg-Weyl operators an index theorem has been proven in \cite{SaSch2}.
The main interest in  these algebras comes from the fact that, here, explicit computations are possible, in particular for index problems in operator algebras associated with groups of quantized canonical transformations as they were considered in \cite{SaSchSt4}, \cite{SaSch1} by the authors and B. Sternin or in  \cite{GKN} by Gorokhovsky, de Kleijn and Nest.

The main analytic object in this article are trace expansions of three types. 
We fix an auxiliary operator $H$ of positive order in the Shubin calculus (a good choice is the harmonic oscillator $H_0=\frac12(|x|^2-\Delta$)), so that the resolvent $(H-\gl)^{-1}$ exists for large $\gl$ in a sector about the negative real axis. 
Given an operator $D=R_gT_wA$ as above, the operator $D(H-\lambda)^{-K}$ will be of trace class whenever $K$ is  sufficiently large, and the trace will be a holomorphic function of $\lambda$. We will show that, as $\gl\to \infty$ in the sector, we have an expansion  
\begin{eqnarray}\label{resolvent_exp}
\Tr(R_gT_wA(H-\lambda)^{-K}) 
&\sim& \sum_{j=0}^\infty c_j (-\lambda)^{(2m+\ord A-j)/\ord H-K} \\
&&+\sum_{j=0}^\infty (c_j' \ln (-\lambda)+c_j'')(-\lambda)^{-j-K}\nonumber
\end{eqnarray}
with suitable coefficients $c_j, c_j', c_j''$, $j=0,1,\ldots$. The trace is $O(|\lambda|^{-\infty})$, if the fixed point set of the affine map $u\mapsto gu+w$ is empty. Here, $\ord A $ and $\ord H$ denote the orders of $A$ and $H$, respectively, and $m$ is the complex dimension of the fixed point set of $g$. The present expansion differs from that in the standard situation, where the dimension of the underlying space enters instead of $m$. 
The coefficients $c_j$ and $c_j'$ are `local': They are determined from the homogeneous components in the asymptotic expansion of the symbols of $D$ and $H$, respectively, and therefore can be computed explicitly,  while the $c_j''$ are `global', they also depend on the residual part of the symbol.
When $g=1$ and $w=0$ we have $m=n$ and recover (an analog of) the classical  result by Grubb and Seeley \cite[Theorem 2.7]{GS1}. 

Following Grubb and Seeley \cite{GS2}, we deduce  two other results. Assuming that $(H-\gl)^{-1}$ exists everywhere in the closed sector, including $\gl=0$, we can define the operator zeta function $\zeta_{R_gT_wA}(z) = \Tr(R_gT_wAH^{-z})$. It is  holomorphic for $\re z>(2n+\ord A)/\ord H$ and extends  meromorphically   to the whole complex plane with at most simple poles. We have the pole structure 
\begin{eqnarray}\label{trace_exp}
\lefteqn{\Gamma(z)\zeta_{R_gT_wA}(z)}\\
&\sim &\sum_{j=0}^\infty \frac{\tilde c_j}{z-(2m+\ord A-j)/\ord H}
+ \sum_{j=0}^\infty \Big(\frac{\tilde c_j'}{(z+j)^2} + \frac{\tilde c_j''}{z+j}\Big)
\nonumber
\end{eqnarray}
with suitable coefficients $\tilde c_j, \tilde c_j'$ and $\tilde c_j''$ related to those above by universal constants. 

If the sector in which the resolvent is holomorphic is larger than the left half plane, the `heat trace'  $\Tr\big(R_gT_wAe^{-tH}\big)$ exists for all $t>0$ and has  an expansion 
\begin{eqnarray}
\label{exp_exp}
\mbox{\;\;\;}\Tr( R_gT_wAe^{-tH})  \sim \sum_{j=0}^\infty \tilde c_j t^{(j-2m-\ord A)/\ord H} + 
\sum_{j=0}^\infty (-\tilde c'_j\ln t + c_j'')t^j
\end{eqnarray}
with the same coefficients as in \eqref{trace_exp}.

After multiplication by $\ord H$ the coefficients $c_0'$ and $\tilde c_0'$ are independent of the auxiliary operator $H$ and yield an analog of the noncommutative residue for our algebra. 
Restricted to the Shubin pseudodifferential operators it coincides with the residue introduced by 
Boggiatto and Nicola \cite{BoNi} for a more general class of anisotropic pseudodifferential operators on $\R^n$. 

Given a discrete subgroup
$G$ of $\C^n\rtimes \U(n)$ and an element $(w_0,g_0)$ of $G$, we define the residue trace 
of the operator $D=\sum R_gT_wA$  localized at the conjugacy class $\skp{(w_0,g_0)}$ by 
\begin{eqnarray}
\label{eq_res}
\res_{\skp{(w_0,g_0)}} D = \ord H\sum_{(w,g)\in \skp{(w_0,g_0)}} c_0'(R_gT_wA)
\end{eqnarray}
and show that these are actually traces on the algebra $\mathscr A$. 

In this sense, our article is strongly related to the work of Dave \cite{D1}, who  defined  equivariant noncommutative residues for the algebra $\Psi_{\rm cl}(M)\rtimes \gG$ of operators on a closed manifold $M$, generated by the classical pseudodifferential operators and a finite group $\gG$ of diffeomorphisms of $M$ and computed the cyclic homology in terms of the de Rham cohomology of the fixed point manifolds $S^*M^g$. 
It is also related to Perrot's paper \cite{Perrot14}, where he studied the local index theory for shift operators associated to non-proper and non-isometric actions of Lie groupoids and  computed the residue of the corresponding operator zeta functions in zero.   

{\bf Structure of the article.} We start by recalling the operator classes in Section \ref{sect_algebra}. 
The main work then goes into the derivation of the expansion \eqref{resolvent_exp} in Theorem \ref{thm_resolvent} with the help of the weakly parametric calculus of Grubb and Seeley \cite{GS1} and stationary phase expansion techniques.  
From Theorem \ref{thm_resolvent} we deduce  Theorem \ref{thm_zeta} on the structure of the singularities of the operator zeta function and Theorem \ref{thm_exp} on the `heat trace' expansion by abstract arguments, following Grubb and Seeley \cite{GS2}. 
In Section 5 we introduce the noncommutative residue. We show that it is explicitly computable and derive the corresponding expression. It allows us to define an equivariant trace on 
$\scrA$ for every conjugacy class in the subgroup  $G$ of $\C^n\rtimes \U(n)$. 
The appendix contains the essential material on the weakly parametric calculus of Grubb and Seeley.
.

{\bf Acknowledgments.} ES gratefully acknowledges the support of Deutsche Forschungsgemeinschaft through grant SCHR 319/10-1.  AS is grateful for support to the Russian Foundation for Basic Research, project Nr.~21-51-12006.
The results in this article were announced in the proceedings contribution \cite{SaSch4}.

\section{Operator Classes}\label{sect_algebra}

\subsection*{Shubin type pseudodifferential operators}
A Shubin type pseudodifferential symbol of order $\fa\in \R$ is a smooth function $a$ on 
$\R^n\times \R^n$ satisfying 
$$|D^\ga_pD^\gb_x a( x,p)| \le C_{\ga,\gb} (1+|x|^2+|p|^2)^{(\fa-|\ga|-|\gb|)/2}, \quad(x,p)\in \R^n\times \R^n,$$
 for all multi-indices $\ga$, $\gb$. 
From $a$ we define the associated pseudodifferential operators $\op a$ and 
$\op^{\rm w} a$, respectively,  by standard and Weyl quantization.

We denote by $\Psi^\fa$ the space of all Shubin type pseudodifferential operators of order $\fa$ and let $\Psi=\bigcup_{\fa\in \R}$.
An important example is the operator
\begin{eqnarray}\label{def_ham}%
H_0 = \frac12(|x|^2-\Delta) \in \Psi^2, 
\end{eqnarray}
the so-called harmonic oscillator, which is selfadjoint and positive on $L^2(\R^n)$. 

The Sobolev spaces $\mathcal H^s(\R^n)$, $s\in \R$, consist of all tempered distributions $u$ such that $\op((1+|x|^2+|p|^2)^{s/2})u\in L^2(\R^n)$. We have $\Psi^\fa\subset  \scrB(\cH^s(\R^n),\cH^{s-\fa}(\R^n)) $ for all  $s\in \R$. In particular, zero order  operators 
are bounded and those of  order $<-2n$ are of trace class on $L^2(\R^n)$.

We will use mostly classical symbols, i.e.~symbols $a$ of order $\fa\in \Z$ that have an asymptotic expansion 
$$a \sim \sum_{j=0}^\infty a_{\fa-j}  $$
with $a_{\fa-j}$ smooth and positively homogeneous in $(x,p)$ of degree $\fa-j$ for 
$|(x,p)|\ge 1$. We call $a_\fa$ the principal symbol and write $\Psi^\fa_{\rm cl} $ 
for the class of operators with classical symbols of order $\le \fa$ and $\Psi_{\rm cl}$ for 
$\bigcup_\fa \Psi_{\rm cl}^\fa$. See \cite{Shu1} for more details.

The above operators need not be scalar, it is often important to also consider systems of operators  over $\R^n$, see e.g. \cite{SaSch3}.

\subsection*{Heisenberg-Weyl operators.} Given  $w=a-ik\in\mathbb{C}^n$,  $a,k\in \mathbb{R}^n$,  define
\begin{eqnarray}
T_w u(x)=e^{ikx-iak/2}u(x-a), \quad u\in \scrS(\R^n).
\end{eqnarray}
For $v,w\in \C^n$ we obtain
\begin{equation}
\label{eq-comm1}
T_{v}T_{w}=e^{-i\im( v,w)/2}T_{v+w},\quad \text{ where } (v,w)=v\overline{w}.
\end{equation}
Moreover, for a Shubin symbol $b$,
\begin{eqnarray}
\label{eq-comm.Sh}
&(T_w)^{-1} \op(b) T_w = T_{-w} \op(b) T_w = \op(\tilde b), \; \tilde b(x,p) = b(x+a,p+k).  
\end{eqnarray}
 
\subsection*{Metaplectic operators.} 
The group $\Mp(n)$ of metaplectic operators on $L^2(\R^n)$ is generated by the operators $S_q=e^{i\op^{\rm w} q}$, where $\op^{\rm w}q$ is the Weyl quantization of a real  homogeneous quadratic form $q=q(x,p)$ on $\mathbb \R^n\times\R^n$. 
The operator $S_q$ induces a symplectic map on $T^*\R^n$ by evaluating the Hamiltonian flow generated by the function $q$ at time $1$. This yields a surjection, in fact a double covering, 
$\pi: \Mp(n)\to \Sp(n)$ of the symplectic group. 

The complex metaplectic group $\Mp^c(n)$ has $e^{i\phi}$, $\phi\in \R$, as additional generators. 
It was shown in \cite[Proposition 1]{SaSch2} that the map 
$$R: \U(n)\to \Mp^c(n),  g \mapsto R_g = \pi^{-1}(g) \sqrt{\det g},$$
first defined for $g$ close to the identity $g=I\in \U(n)$ with the choices $\pi^{-1} (I)=I$ and $\sqrt{\det I}=1$, extends continuously to all of $\U(n)$.  These are the operators used above. The metaplectic operators form a special class of Fourier integral operators, compatible with the Shubin pseudodifferential operators in the sense that, for a Shubin symbol $a$ and a metaplectic operator $S$,    
\begin{eqnarray}
\label{Egorov}
S^{-1} \op^{\rm w}(a) S = \op^{\rm w}(a\circ \pi(S)). 
\end{eqnarray}

Since $\U(n)$ is generated by ${\rm O}(n)$ and $\U(1)$, see~\cite[Lemma 1]{SaSch2},
the unitary representation $g\mapsto R_g$ of $\U(n)$ is determined  by letting 
\begin{enumerate}
\item $R_gu(x)=u(g^{-1}x)$, if $g\in {\rm O}(n)\subset \U(n)$. \item $R_gu(x)=e^{i\varphi(1/2-H_1)}u(x)$, where $g={\rm diag}(e^{i\varphi},1,\ldots,1),$ and $H_1=\frac12(x_1^2-\partial_{x_1}^2)$ is the one-dimensional  harmonic oscillator. 
\end{enumerate} 
In the first case,  $R_g$  is a so-called shift operator; in the second, it is
a fractional Fourier transform with respect to $x_1$. According to Mehler's formula \cite[Section 4]{H95}, 
\cite{MK87}, 
\begin{eqnarray*}\lefteqn{R_gu(x)
=\sqrt{\frac{1-i\ctg \varphi}{2\pi}}}\\
&&\times \int 
   \exp\left(
        i\left(
         (x_1^2+y_1^2)\frac{\ctg \varphi}2-\frac{x_1y_1}{\sin\varphi} 
        \right) 
       \right) u(y_1,x_2,...,x_n)dy_1.
\end{eqnarray*}
Here we choose the square root with positive  real part. 
More precisely, we set
\begin{equation}
\label{eq-mehler1}
\sqrt{ 1-i\ctg \varphi}=\frac{\exp(-i(\frac\pi 4\sgn\varphi-\frac\varphi 2))}{|\sin\varphi|^{1/2}},\qquad
0<|\varphi|<\pi.
\end{equation}

In (ii) we may  confine ourselves to  $-\pi<\gvp<\pi$, $\gvp\not=0$, as the cases $\gvp=0,\pi$ are covered by (i). For $\gvp=0$,  $R_g$ is the identity,  for $\gvp=\pi$,  it is the reflection in the first component and  for $\gvp=\pi/2$,  
it is the Fourier transform in the first variable. 
By Theorem~7.13 in~\cite{deG1} we have 
\begin{equation} 
\label{cr12}
R_g T_w R_g^{-1}=  T_{g w}
\end{equation}   
for all $g\in \U(n)$, $w\in \C^n$. 
The relations \eqref{eq-comm.Sh}, \eqref{Egorov} and \eqref{cr12} imply that the operators of the form \eqref{algebra} indeed form an algebra. 

As a consequence of \eqref{Egorov} we obtain: 
\begin{lemma}\label{mapping} 
A metaplectic operator $S$ defines a continuous linear map 
$$S: \cH^s(\R^n) \longrightarrow \cH^s(\R^n), \ s\in \R,\quad \text{ and }\quad 
S: \scrS(\R^n) \longrightarrow \scrS (\R^n). $$ 
\end{lemma} 

\begin{proof} The metaplectic operators are unitaries on $L^2$. 
Let $u\in \cH^2$.  Clearly,
$$\cH^2=\{u\in L^2: H_0u\in L^2\}$$
with the harmonic oscillator $H_0$. 
So we simply observe that $Su\in L^2$ and that, according to \eqref{Egorov}, 
$$H_0Su = S(S^{-1}H_0S)u = S\op^{\rm w}(h_0\circ\pi(S^{-1}))u\in L^2.$$ 
Here, $h_0$ is the Weyl symbol of $H_0$, and the last inclusion holds, since 
$\op^{\rm w} (h_0\circ\pi(S^{-1}))\in \Psi^2$ and $u\in \cH^2$. 
Hence $S:\cH^2 \to \cH^2$ is bounded. 
Similarly, $u\in \cH^{-2}$ implies that $u=(I+H_0)u_0$ for some $u_0\in L^2$, since 
$I+H_0:L^2\to \cH^{-2}$ is an isomorphism. 
Hence 
$$Su = (S(I+H_0)S^{-1})S u_0 = \tilde H Su_0\in \cH^{-2}(\R^n), $$  
since $Su_0\in L^2$ and $\tilde H = S(I+H_0)S^{-1}\in \Psi^{2}$ is bounded
$L^2\to \cH^{-2}$.  
So, $S: \cH^{-2}\to \cH^{-2}$ is also bounded. Iteration and interpolation show the assertion for all $s\in \R $. The continuity on $\scrS$ follows, since $\scrS$ is the intersection of all $\cH^s$, $s\in \R$.  
\end{proof}


\section{The Resolvent Expansion}\label{ResolventExpansion}


Let $E$ be a (trivial) vector bundle over $\R^n$ and $H=\op (h) \in \Psi^\fm_{\rm cl}$ with $\fm=\ord H >0$, a pseudodifferential operator acting on sections of $E$. 
We assume that the  principal symbol $h_\fm$ of $h$ is  a Shubin symbol of order  $\fm$,   
homogeneous for $|(x,p)|\ge 1$, scalar (i.e.~a multiple of $id_E$) and strictly positive.

For  $0<\delta<\pi$ denote by $S_{\delta}$  the sector 
\begin{eqnarray}\label{sector}%
S_{\delta} = \{\lambda\in \mathbb C\setminus \{0\}:  |\arg\lambda - \pi |<\delta\}
\end{eqnarray}
and by $U_r(0)$  the disc of radius  $r>0$ about $0\in \C$. 
Standard pseudodifferential techniques show: 

\begin{proposition}\label{prop_resolvent}
For all $\lambda \in S_\gd$ with  $|\lambda| $ sufficiently large, $H-\lambda$ is invertible. 
The function $\lambda\mapsto (H-\gl)^{-1}$ extends meromorphically to 
$\C$, taking values in $\Psi^{-\fm}_{\rm cl}$.

In particular, we may assume that $H-\lambda$ is invertible for all 
$\gl \in S_\gd$ after replacing $H$ by $H+c$ for suitable $c>0$.  
\end{proposition} 

The following theorem is the crucial analytic result:

\begin{theorem}\label{thm_resolvent} 
For $g\in \U(n)$, $w=a-ik\in \C^n$, $A\in \Psi^\fa_{\rm cl}$, $\fa=\ord A\in \mathbb Z$,   
and  $K\in \N$ so large that $-K\fm +\fa<-2n$ we have 
\begin{enumerate}\renewcommand{\labelenumi}{(\alph{enumi})} 
\item The operator $R_gT_w A (H-\lambda)^{-K} $ is of trace class for all $\lambda$ in  $S_\gd$, $|\gl|$ sufficiently large. 

\item  The function $\lambda \mapsto \Tr(R_gT_w A (H-\lambda)^{-K}) $ is meromorphic in $S_\gd$  and near $\gl=0$. It is $O(|\gl|^{-\gve})$ with suitable $\gve>0$ for $\gl\to \infty$ in $S_\gd$. 

\item As $\gl\to \infty$ in the sector $S_\gd$ we have the expansion \eqref{resolvent_exp}, i.e. 
\begin{eqnarray*}%
\Tr(R_gT_w A (H-\lambda)^{-K}) 
&\sim& \sum_{j=0}^\infty c_j (-\lambda)^{(2m+\ord A-j)/\ord H-K} \\
&&+\sum_{j=0}^\infty (c_j' \ln (-\lambda)+c_j'')(-\lambda)^{-j-K}
\end{eqnarray*}
with suitable coefficients $c_j, c_j', c_j''$, $j=0,1,\ldots$.  
Here, $m=\dim_\C(\C^n)^g$  is the complex dimension of the fixed point set of $g$.  

\item If the affine mapping on $\mathbb{C}^n$ given by $v\mapsto gv+w$ has no fixed points, 
then the trace is $O(|\gl|^{-\infty})$. 

\item The product $\fm c_0'$ is independent of the choice of $H$.
\end{enumerate} 
\end{theorem} 

The proof of Theorem \ref{thm_resolvent} will be given in several steps and occupies the rest of the section. 
Clearly,  (a) and (b) follow from Proposition \ref{prop_resolvent} and the fact that the embedding 
$\cH^{2n+\gve}(\R^n) \hookrightarrow L^2(\R^n)$ 
is trace class for every $\gve>0$. For the remaining items, we observe:

\begin{lemma}\label{g}
We may assume that $g$ is diagonal and of the form 
\begin{equation} \label{eq-diag2}
g={\rm diag}\big(\underbrace{e^{i\varphi_1},...,e^{i\varphi_{m_1}}}_{m_1},\underbrace{i,...,i}_{m_2},\underbrace{-i,...,-i}_{m_3},\underbrace{-1,...,-1}_{m_4},\underbrace{1,...,1}_{m_5}\big)\in \U(n),
\end{equation}
where $\varphi_j\in {]-\pi,\pi[}\setminus  \pi\mathbb{Z}/2$ and   $m_5=\dim_\C(\mathbb{C}^n)^g$.   

{\rm For notational convenience we will also write  $\varphi_j=\pi/2$ for 
$j=m_1+1, \ldots, m_1+m_2$ and $\varphi_j = -\pi/2$ for $j= m_1+m_2+1,\ldots, m_1+m_2+m_3$. }
\end{lemma} 

{\em Proof of Lemma} 3.3.
Since $g$ is diagonalizable, there exists   $u\in \U(n)$ such that $g=ug_0u^{-1}$ with $g_0$ diagonal and unitary. Then
\begin{eqnarray*}\lefteqn{
\Tr(R_gT_wA H^{-z})=\Tr(R_uR_{g_0}R_u^{-1}T_wA H^{-z})}\\
&=&\Tr( R_{g_0}(R_u^{-1}T_wR_u)(R_u^{-1} A  R_u)  (R_u^{-1}H^{-z}R_u))
=\Tr( R_{g_0}T_{w'} A' (H')^{-z}).
\end{eqnarray*}
Here $A'=R_u^{-1} A  R_u\in \Psi^\fa$ by \eqref{Egorov}, $R_u^{-1}T_wR_u= T_{w'}$ with $w'=u^{-1}w$ by \eqref{cr12}, and  $H'=R_{u^{-1}}HR_{u}$
is an admissible auxiliary operator of the same order as $H$, since the principal symbol is also scalar and positive by \eqref{Egorov}. 
(In the special case, where $H=H_0$ is the harmonic oscillator of \eqref{def_ham}, we even have $H'=H_0$, since $H_0$ commutes with $R_u$.) 
On the diagonal, $g_0$ has entries $e^{i\varphi}$.  As noted at the end of Section \ref{sect_algebra}, four values for 
$\varphi$ are special, namely $\varphi=0, \pm\pi/2, \pi$. 
They are listed according to their 
multiplicities in \eqref{eq-diag2}.  
\hfill $\Box$ 

In order to show (c), (d) and (e), we 
write $-\lambda = \mu^{\fm}$ for $\mu$ in the sector  
$$S=\{z\in \C\setminus \{0\}: |\arg z|<\delta/\fm\}.$$ 
So $h_\fm(x,p)+\mu^\fm$ is invertible for $\mu \in S$ and homogeneous of degree $\fm$
in $(x,p)$ and $\mu$  for $|(x,p)|\ge1$. We will then find  the  expansion in terms of $\mu$.    

By $q(y,p;\mu)$ denote the complete symbol of $A(H+\mu^\fm)^{-K}$ in $y$-form,  
i.e. 
\begin{eqnarray}\label{AHmu}
\Big(A(H+\mu^\fm)^{-K} \Big)u(x) = (2\pi)^{-n} \int e^{i\skp{x-y,p}} q(y,p;\mu) u(y) dydp. 
\end{eqnarray}
This yields the following expression for the trace: 

\begin{proposition}\label{prop_trace}
Up to terms that are $O(|\mu|^{-\infty})$ as $\mu\to \infty $ in $S$, \begin{eqnarray*}%
 \Tr(R_gT_w A(H+\mu^\fm)^{-K}) = 
 C_{\rm res} \int e^{i\phi(u,v)}\tr_E q(B(u,v) -b_0, y,\eta;\mu)\, dud\underline{v}dyd\eta.
\end{eqnarray*}
Here, 
$$C_{\rm res} = (2\pi)^{-n+\frac{m_2+m_3}2} 
\prod_{j=1}^{m_1}  \sqrt{1+i\tg\varphi_j} 
   \prod_{j=1}^{m_1+m_2+m_3}
 e^{\frac i4\ctg(\varphi_j/2)(k_j^2+a_j^2)}$$ 
(we take the square root with positive real part)
is a constant, 
$\tr_E$ is the trace in the vector bundle $E$, and 
$$(x_1,p_1, \ldots ,x_{n-m_5},p_{n-m_5})= B(u,v)-b_0,
\quad  u,v\in\mathbb{R}^{n-m_5},
$$ 
is an affine linear change of coordinates. 

As a mapping $ (u_j,v_j)\mapsto(x_j,p_j) $, $j=1, \ldots, n-m_5$, the matrix $B$ is a diagonal block matrix with entries $\frac1{\sqrt2}\mat{1&1\\-1&1}$ for  $j\not=m_1+1,\ldots,m_1+m_2+m_3$, for the latter it is the identity. 
In particular, $B$ is orthogonal. 
Moreover, 
$$b_0 = (u_{0,1}, v_{0,1}, \ldots, u_{0,n-m_5}, v_{0,n-m_5})$$  
with
$$(u_{0,j},v_{0,j} ) = \frac12\left(a_j-k_j\ctg\frac{\gvp_j}2, a_j\ctg\frac{\gvp_j}2+k_j\right). 
$$
Finally 
$y=(x_{n-m_5+1},\ldots, x_n)$, $\eta=(p_{n-m_5+1},\ldots ,p_n)$,
and `$d\underline v$' indicates that the integration is not over $v_{m_1+1}, \ldots ,v_{m_1+m_2+m_3}$; for these variables, $v_j$ is evaluated at the  points
$v_j = \sin\gvp_ju_{0,j}$.

The phase $\phi$ is given by 
\begin{eqnarray}%
\phi(u,v)&=&
-\sum_{j=1}^{m_1} (\lambda^-_ju_j^2 +\lambda^+_jv_j^2)-\sum_{j=m_1+1}^{m_1+m_2}u_j^2    
+\sum_{j=m_1+m_2+1}^{m_1+m_2+m_3}u_j^2  \nonumber \\        
&&-\sum_{j=m_1+m_2+m_3+1}^{n-m_5}(-u_j^2+v_j^2)\label{eq_phase}
\end{eqnarray}
with 
\begin{eqnarray}\label{lambdaj}
\gl_j^\pm  
= \frac{1}2  \frac{\sin \gvp_j \mp(1-\cos\gvp_j)}{\cos\gvp_j}.
\end{eqnarray}
\end{proposition} 

It will turn out that the expansion is $O(|\mu|^{-\infty})$ whenever $a_j\not=0$ or $k_j\not=0$ for some $j>n-m_5$. This will show (d). We will thus assume $a_j=k_j=0$ for these $j$. 
This is also the reason why  $b_0$ has no entries for $j>n-m_5$.

\begin{proof}
The proof is  divided  into several steps. 

{\em Step} 1. 
From \eqref{AHmu} we obtain the continuous Schwartz kernel 
\begin{eqnarray*}%
K_{T_wA(H+\mu^\fm)^{-K}}(x,y) = (2\pi)^{-n} e^{-i\skp{a,k}/2} 
\int e^{i\skp{x-a-y,p}+i\skp{k,x}} q(y,p;\mu)\, dy dp.
\end{eqnarray*}
Mehler's formula for $R_g$ shows that 
\begin{eqnarray*}
\lefteqn{K_{R_gT_wA(H+\mu^\fm)^{-K}}(x,y)}\\
&=& 
 (2\pi)^{-n} e^{-i\skp{a,k}/2}\prod_{j=1}^{m_1+m_2+m_3} 
\sqrt{\frac{1-i\ctg\varphi_j}{2\pi}}
\int  e^{i\phi_1(x,y,p,y')} q(y,p;\mu) \, dp dy'
\end{eqnarray*}
with $y' = (y'_1,\ldots, y'_{m_1+m_2+m_3})$ and 
\begin{eqnarray*}
\lefteqn{\phi_1(x,y,p ,y')}\\  
&=&\sum_{j=1}^{m_1+m_2+m_3}
\Big(
-\frac{x_jy_j'}{\sin\varphi_j} + \frac{\ctg \varphi_j}{2}(x_j^2+{y_j'}^2 )
+ (y_j' -a_j-y_j)p_j +k_jy_j'
\Big) \\   
&&+\sum_{j=m_1+m_2+m_3+1}^{n-m_5}((-x_j -a_j-y_j)p_j-k_jx_j )\\ 
&&+\sum_{j=n-m_5+1}^n((x_j -a_j-y_j)p_j+k_jx_j ).
\end{eqnarray*}

{\em Step} 2. 
The trace of  $R_gT_wA(H+\mu^{\fm})^{-K}$  is  the integral of the Schwartz kernel over the diagonal. As 
$\ctg\gvp_j=0$ for $\gvp_j=\pm\frac\pi2$, we obtain 
\begin{eqnarray}\label{TrRTA}
\Tr(R_gT_wA(H+\mu^{\fm})^{-K})&=& (2\pi)^{-n-(m_2+m_3)/2} e^{-i\skp{a,k}/2}\prod_{j=1}^{m_1}
\sqrt{\frac{1-i\ctg\varphi_j}{2\pi}}\nonumber\\ 
&&\times 
\int e^{i\phi_2(x,p,y')} \tr_Eq(x,p;\mu)\, dpdy'dx
\end{eqnarray}
with 
\begin{eqnarray*}%
\lefteqn{\phi_2(x,p,y')}\\  &=&
\sum_{j=1}^{m_1+m_2+m_3}\left(-\frac{x_jy_j'}{\sin\varphi_j}+\frac{\ctg\varphi_j}2(x_j^2+{y_j'}^2)
+(y_j'-a_j-x_j)p_j+k_jy_j'\right)\\
&&-\sum_{j=m_1+m_2+m_3+1}^{n-m_5}((2x_j +a_j)p_j+k_jx_j )+\sum_{j=n-m_5+1}^n( -a_j p_j+k_jx_j ).
\end{eqnarray*}

If  $a_j\ne 0$ or $k_j\ne 0$ for some  $j>n-m_5$ 
(i.e., if the affine mapping 
$z\mapsto gz+w$ has no fixed points) 
then integration by parts in \eqref{TrRTA} shows that the trace is 
$O(\mu^{-\infty})$. 
This proves statement (d) in Theorem \ref{thm_resolvent}.
These terms do not contribute to the asymptotic expansion, 
and so we shall assume in the sequel  that  $a_j=k_j=0$ for all $j>n-m_5.$

{\em Step} 3. The formula for  Gaussian integrals
\begin{eqnarray*}%
\int e^{-ax^2+bx+c}\, dx = \sqrt{\frac\pi a} e^{b^2/4a +c}
\end{eqnarray*}
and the fact that the one-dimensional Fourier transform 
$\frac1{\sqrt{2\pi}} \int e^{-i\omega t} dt$ of the constant function $1$ is $\sqrt{2\pi}\delta$ imply that the  integrals over $y_j'$ in \eqref{TrRTA} are 
\begin{multline*}
\int  \exp \Biggl(i 
\left( \frac{\ctg\varphi_j}2 {y_j'}^2 
+ y_j'\left( p_j+k_j -\frac{x_j }{\sin\varphi_j}\right)\right)\Biggr)dy'_j\\
=
\left\{
\begin{array}{ll}
\displaystyle \sqrt{\frac{2\pi}{-i\ctg\varphi_j}}\ \exp \Bigl(-\frac i 2 
\tg\varphi_j\Bigl(p_j+k_j-\frac{x_j }{\sin\varphi_j}\Bigr)^2\Bigr), 
& \text{if }\varphi_j\notin \pi\mathbb{Z}/2,\\
\displaystyle 2\pi \;\;\delta\left(p_j+k_j- \frac{x_j}{\sin\varphi_j}\right), & \text{if }\varphi_j=\pm\pi /2.
\end{array}
\right.
\end{multline*}
Here we take the  branch of the square root determined by
$$
\sqrt{\frac{2\pi}{-i\ctg\varphi_j}}=
\sqrt{\frac{2\pi}{|\ctg\varphi_j|}}e^{i\frac\pi 4\sgn\ctg\varphi_j}.
$$
This formula can be obtained from the Gaussian integral inserting a factor $e^{-\varepsilon {y'_j}^2}$ and considering the limit $\varepsilon\to 0^+$.

For $p=(p',p'')$ with  $p'=(p_1,\ldots,p_{m_1},p_{m_1+m_2+m_3+1},\ldots,p_n)$ and $p''=(p_{m_1+1},\ldots ,p_{m_1+m_2+m_3})$,  integration of the $\delta$-functions for  $\gvp_j = \pm \pi/2$ yields
$$p_j= \frac{x_j}{\sin\varphi_j} -k_j= \pm x_j-k_j, \quad j=m_1+1,\ldots,m_1+m_2+m_3.$$
Inserting this into Equation \eqref{TrRTA} we find that 
\begin{eqnarray*}
\lefteqn{
\Tr(R_gT_wA(H+\mu^{M})^{-K})}\\ 
&=&  (2\pi)^{-n+\frac{m_2+m_3}2} e^{-i\skp{a,k}/2}
\prod_{j=1}^{m_1}\Big( \sqrt{\frac{1-i\ctg\varphi_j}{2\pi}}
\sqrt{\frac{2\pi}{-i\ctg\varphi_j}} 
\Big)\\ 
&&\times \int e^{i\phi_3(x,p')} \tr_Eq(x,p',p'';\mu)_{|p''_j = \pm x_j-k_j}\, dp'dx.
\end{eqnarray*}
Here we have  (see Mehler's formula and Gaussian integration above)
\begin{multline*}
\sqrt{ 1-i\ctg\varphi_j } \sqrt{\frac{1}{-i\ctg\varphi_j}}=
\frac{\exp(-i(\frac\pi 4\sgn\varphi_j-\frac{\varphi_j} 2))}{|\sin\varphi_j|^{1/2}}
\sqrt{\frac{1}{|\ctg\varphi_j|}}e^{i\frac\pi 4\sgn\ctg\varphi_j}\\
=\frac{e^{i\frac{\varphi_j} 2}}{|\cos\varphi_j|^{1/2}}e^{i\frac\pi 4(-\sgn\varphi_j+\sgn\ctg\varphi_j)} 
=\sqrt{1+i\tg\varphi_j},
\end{multline*}
where we take the square root with positive real part.

The notation $ q(x,p',p'';\mu)_{|p''_j = \pm x_j-k_j}$ indicates that we 
evaluate $p''_j$ at $x_j-k_j$ for $j=m_1+1,\ldots, m_1+m_2$ and at $-x_j-k_j$ 
for $j=m_1+m_2+1,\ldots, m_1+m_2+m_3$. The phase $\phi_3$ is given by
\begin{eqnarray}\label{phase3}
\lefteqn{\phi_3(x,p')=
\sum_{j=1}^{m_1}
     \Bigl(
        x_j^2\Bigl(\frac{\ctg\varphi_j}2-\frac{1}{ \sin 2\varphi_j }\Bigr)
       +x_jp_j\Bigl(\frac{1}{\cos\varphi_j}-1 \Bigr)}\\
       &&-p_j^2\frac{\tg\varphi_j}2  
       +x_j\frac{k_j}{\cos\varphi_j}
       -p_j\left( a_j+k_j\tg\varphi_j\right)-\frac{ \tg\varphi_j}{2}k_j^2
      \Bigr)\nonumber  \\
&&-\sum_{j=m_1+1}^{m_1+m_2+m_3}(x_j+a_j)\left(\frac{x_j}{\sin\varphi_j}-k_j\right)   \nonumber  \\  
&&-\sum_{j=m_1+m_2+m_3+1}^{n-m_5}( 2x_jp_j +a_j p_j+k_jx_j ).\nonumber
\end{eqnarray}
Since 
\begin{eqnarray*}%
\ctg\varphi_j -\frac1{\sin\varphi_j\cos\varphi_j }
&=& \frac{\cos^2\varphi_j -1}{\sin\varphi_j\cos\varphi_j }
=-\tg\varphi_j\text{ and}\\
\frac1{\cos\varphi_j}-1&=& \tg\varphi_j \ \frac{1-\cos\varphi_j}{\sin\gvp_j} ,
\end{eqnarray*}
we see that 
\begin{eqnarray*}
\phi_3(x,p) &=& 
-\sum_{j=1}^{m_1} \frac{\tg\gvp_j}2
\Big(x_j^2\ 
-2x_jp_j\Big(\frac{1-\cos\varphi_j}{\sin\gvp_j}\Big)
+p_j^2 \\
&&-2x_j\frac{k_j}{\sin \varphi_j}
+2p_j\Big(
\frac{a_j}{\tg \gvp_j} 
+k_j\Big)+k_j^2\Big)\\
&&-\sum_{j=m_1+1}^{m_1+m_2+m_3}(x_j+a_j)\left(\frac{x_j}{\sin\varphi_j}-k_j\right) \\      
&&-\sum_{j=m_1+m_2+m_3+1}^{n-m_5}( 2x_jp_j +a_j p_j+k_jx_j ).
\end{eqnarray*}

{\em Step} 4. We shall next make a transformation in the first $m_1$ coordinate pairs $(x_j,p_j)$ of the form 
\begin{eqnarray*}
\mat{x\\p} = B'\mat{u\\v} -b'_0, \qquad b'_0=\mat{u'_0\\v'_0}
\end{eqnarray*}
that results in purely quadratic terms plus constants in the phase: Let 
\begin{eqnarray}
\ga_j = -2\frac{1-\cos\gvp_j}{\sin\gvp_j}, \quad  
\gb_j =-\frac{2k_j}{\sin\gvp_j},\quad 
\gg_j = \frac{2a_j}{\tg \gvp_j}+2k_j.\label{gammaj}
\end{eqnarray}

Due to the special structure of $\phi_3$ we can do this separately for all coordinate pairs $(x_j,p_j)$. In the following short computation 
we therefore suppress the subscripts $j$ and the prime. 
We let
\begin{eqnarray}\label{change}%
\mat{x\\p} = \frac1{\sqrt2} \mat{1&1\\-1&1}\mat{u\\v} -\mat{u_0\\v_0}.
\end{eqnarray}
Then 
\begin{eqnarray*}%
x&=&\frac1{\sqrt2}(u+v)-u_0\\
p&=&\frac1{\sqrt2}(-u+v)-v_0
\end{eqnarray*}
and
\begin{eqnarray*}\lefteqn{
( x^2+\ga xp + p^2) + \gb x + \gg p}\\
&=&\frac12\Big( (2-\ga)u^2 +(2+\ga)v^2 \Big)\\
&&- \frac u{\sqrt2}((u_0-v_0)(2-\ga)-\gb+\gg)\\
&&-\frac v{\sqrt2}((u_0+v_0)(2+\ga)-\gb-\gg) \\
&&+u_0^2+v_0^2+\ga u_0v_0 -\gb u_0 -\gg v_0.
\end{eqnarray*}
As we want the coefficients of $u$ and $v$ to vanish, we obtain the equations 
\begin{eqnarray*}%
u_0-v_0= \frac{\gb-\gg}{2-\ga}\ \text{ and } \ 
u_0+v_0= \frac{\gb+\gg}{2+\ga},
\end{eqnarray*}
and hence 
\begin{eqnarray*}%
u_0&=& \frac12 \left(\frac{\gb+\gg}{2+\ga}+\frac{\gb-\gg}{2-\ga} \right) 
= \frac{2\gb-\gg\ga}{4-\ga^2}= \frac12\left(a-k\frac{1+\cos \gvp}{\sin \gvp}\right) \label{u0}
= \frac{a-k\ctg(\gvp/2)}2;\\
v_0&=& \frac12 \left(\frac{\gb+\gg}{2+\ga} - \frac{\gb-\gg}{2-\ga}\right) 
= \frac{2\gg-\gb\ga}{4-\ga^2} = \frac12\left(a\frac{1+\cos \gvp}{\sin \gvp}+k \right)=
\frac{a\ctg(\gvp/2)+k}2\label{v0}.
\end{eqnarray*}

Changing coordinates from $(x_j,p_j)$ to $(u'_j,v'_j)$ according to 
\eqref{change}, we obtain 
\begin{eqnarray*}%
\lefteqn{
\Tr(R_gT_wA(H+\mu^{\fm})^{-K})}\\ 
&=&  (2\pi)^{-n+\frac{m_2+m_3}2} e^{-i\skp{a,k}/2}
\prod_{j=1}^{m_1} 
\Big(\sqrt{1+i\tg\varphi_j}\;e^{i(\delta_j-(k^2_j/2)\tg\varphi_j)}\Big)\\
&&\times
\int e^{i\phi_4(u',v',x_{m_1+1},\ldots,x_{n-m_5},p_{m_1+m_2+m_3+1}\ldots, p_{n-m_5} )} 
\\
&&\times q(B'\mat{u\\v}-b'_0,x_{m_1+1},\ldots,x_n,p_{m_1+1},\ldots,p_n;\mu)_{|p_j= \pm x_j-k_j, j=m_1+1,\ldots,m_1+m_2+m_3}\\
&&\times \, du_1\ldots du_{m_1}dv_1\ldots dv_{m_1} dx_{m_1+1} \ldots dx_n
dp_{m_1+m_2+m_3+1} \ldots dp_n
\end{eqnarray*}
with $B'$ resulting from the matrices $\frac1{\sqrt2} \mat{1&1\\-1&1}$ above,   
\begin{eqnarray*}%
\phi_4 &=& -\sum_{j=1}^{m_1} (\lambda^-_ju_j^2 +\lambda^+_jv_j^2)
-\sum_{j=m_1+1}^{m_1+m_2+m_3}(x_j+a_j)\left(\frac{x_j}{\sin\varphi_j}-k_j\right)   \\    
&&-\sum_{j=m_1+m_2+m_3+1}^{n-m_5}( 2x_jp_j +a_j p_j+k_jx_j ),
\end{eqnarray*}
\begin{eqnarray*}
\gl_j^\pm =\frac{\tg\varphi_j}4 (2\pm \ga_j )= \frac{1}2  \frac{\sin \gvp_j \mp(1-\cos\gvp_j)}{\cos\gvp_j}
\end{eqnarray*}
and 
\begin{eqnarray}\label{deltaj}%
\delta_j =-\frac{\tg \varphi_j} 2( u_{0j}^2 +v_{0j}^2 +\alpha_j u_{0j} v_{0j} -\gb_ju_{0j} -\gg_jv_{0j}),
\end{eqnarray}
where $\ga_j, \gb_j,\gg_j, u_{0j}, v_{0j}$ are the quantities denoted by $\ga,\gb,\gg,u_0,v_0$ 
above.  

Note that $\gl^\pm_j\not=0$, since 
$\lambda_j^\pm=0$  would require that 
\begin{eqnarray*}%
 \cos \gvp_j + \sin \gvp_j = 1 \text{ or } \cos\gvp_j-\sin\gvp_j=1 
\end{eqnarray*}
or, equivalently, that 
$$\sin(\gvp_j+\frac\pi4) = \frac{\sqrt2}{2} \text{ or } 
\cos(\gvp_j+\frac{\pi}4) = \frac{\sqrt2}{2}.$$ 
This, however, is not the case for $\varphi_j\notin \frac{\pi}{2} \mathbb Z$.

Direct computations show that
\begin{align*}
&u_{0j}^2 +v_{0j}^2 +\alpha_j u_{0j} v_{0j} -\gb_ju_{0j} -\gg_jv_{0j}=
\frac{\alpha_j\beta_j\gamma_j-\beta^2_j-\gamma^2_j}{4-\alpha^2_j}\\
&=\frac{\frac{8(1-\cos\varphi_j)}{\sin^2\varphi_j} (a_jk_j\ctg\varphi_j+k^2_j)
-4k^2\left(\frac1{\sin^2\varphi_j}+1\right)-4a_j^2\ctg^2\varphi_j-8a_jk_j\ctg\varphi_j}{\frac{8\cos\varphi_j}{1+\cos\varphi_j}}\\
&=-4a_j^2\ctg^2\varphi_j-8a_jk_j\frac{\cos^2\varphi_j}{\sin\varphi_j(1+\cos\varphi_j)}
+4k_j^2\frac{\cos\varphi_j(\cos\varphi_j-2)}{\sin^2\varphi_j}.
\end{align*}
Hence,
$$
\delta_j=\frac{a_j^2}{4}\frac{1+\cos\varphi_j}{\sin\varphi_j}+\frac{a_jk_j}{2}
-\frac{k_j^2}4 \frac{(\cos\varphi_j-2)(\cos\varphi_j+1)}{\cos\varphi_j\sin\varphi_j}.
$$
Therefore, we obtain
$$
\delta_j-a_jk_j/2-(k_j^2/2)\tg\varphi_j=(a_j^2+k_j^2)\left(\frac{1+\cos\varphi_j}{\sin\varphi_j} \right)=(a_j^2+k_j^2)\ctg(\varphi_j/2).
$$
{\em Step} 5. 
Next we change coordinates for the $(x_j,p_j)$, $j=m_1+1,\ldots ,m_1+m_2+m_3$, where $\gvp_j=\pm \frac\pi2$.
For 
\begin{eqnarray*}
u_j &=&x_j+\frac{a_j-k_j\sin\gvp_j}{2}= x_j+\frac{a_j\mp k_j}{2},\\
v_j &=&p_j+\frac{a_j\sin\gvp_j+k_j}{2} = p_j +\frac{\pm a_j+k_j}{2}.
\end{eqnarray*}
\begin{eqnarray*}\lefteqn{(x_j+a_j)\left(\frac{x_j}{\sin\gvp_j} -k_j\right) }\\
&=& \frac1{\sin\gvp_j}
\Big( x_j+\frac{a_j-k_j\sin\gvp_j}2\Big)^2
-\frac1{\sin\gvp_j}\Big(\frac{a_j-k_j\sin\gvp_j}{2}\Big)^2
-a_jk_j\\
&=& \frac1{\sin\varphi_j}u_j^2-\frac1{\sin\gvp_j}\Big(\frac{a_j-k_j\sin\gvp_j}{2}\Big)^2-a_jk_j\\
&=& \pm u_j^2\mp\Big(\frac{ a_j\mp k_j}{2}\Big)^2-a_jk_j.
\end{eqnarray*}

We therefore obtain 
\begin{eqnarray}%
\lefteqn{
\Tr(R_gT_wA(H+\mu^{\fm})^{-K})}\nonumber\\ 
&=&  (2\pi)^{-n+\frac{m_2+m_3}2} e^{-i\skp{a,k}/2}
\prod_{j=1}^{m_1} 
\Big(\sqrt{1+i\tg\varphi_j}\;e^{i\delta_j-i(k_j^2/2)\tg\varphi_j}\Big)\nonumber\\
&&\times \prod_{j=m_1+1}^{m_1+m_2}
 e^{i\big(\frac{k_j-a_j}{2}\big)^2+ia_jk_j}  \prod_{j=m_1+m_2+1}^{m_1+m_2+m_3}
 e^{-i\big(\frac{k_j+a_j}{2}\big)^2+ia_jk_j}\nonumber\\
&&\times \int e^{i\phi_5(u,v,x''',p''')} 
\nonumber\\
&&\times 
\tr_Eq(B'\mat{u'\\v'}-b'_0,\mat{u''\\v''}-b_0'',\underline x''',\underline p''';\mu)%
_{|v''_j= \sin \varphi_j u''_j}  
\label{uv}\\
&&\times \, du'dv'du'' d\underline x'''d\underline p'''\nonumber
\end{eqnarray} 
with 
\begin{eqnarray*}%
b''_0 &=&\mat{\frac{a_j-k_j\sin\gvp_j}{2}\\\frac{a_j\sin\gvp_j +k_j}{2}},\\
u'&=&(u_1,\ldots, u_{m_1}), 
v'=(v_1,\ldots, v_{m_1}),\\
u''&=&(u_{m_1+1},\ldots, u_{m_1+m_2+m_3}), 
v''=(v_{m_1+1},\ldots, v_{m_1+m_2+m_3}),\\ 
x'''&=&(x_{m_1+m_2+m_3+1},\ldots,x_{n-m_5}), \\
\underline x'''&=&(x_{m_1+m_2+m_3+1},\ldots,x_{n}),\\ 
p'''&=&(p_{m_1+m_2+m_3+1},\ldots,p_{n-m_5}),\\
\underline p'''&=&(p_{m_1+m_2+m_3+1},\ldots,p_{n}),
\end{eqnarray*}
and the phase 
\begin{eqnarray*}%
\phi_5&=& 
 -\sum_{j=1}^{m_1} (\lambda^-_ju_j^2 +\lambda^+_jv_j^2)
-\sum_{j=m_1+1}^{m_1+m_2}u_j^2    
+\sum_{j=m_1+m_2+1}^{m_1+m_2+m_3}u_j^2   \\    
&&-\sum_{j=m_1+m_2+m_3+1}^{n-m_5}( 2x_jp_j +a_j p_j+k_jx_j ).
\end{eqnarray*}

Note that we can rewrite the factor in front of the integral as 
\begin{eqnarray*}
\lefteqn{(2\pi)^{-n+\frac{m_2+m_3}2} 
\prod_{j=1}^{m_1} 
 \sqrt{1+i\tg\varphi_j}
 \nonumber}\\
&&\times \prod_{j=1}^{m_1+m_2+m_3}
 e^{\frac i4\ctg(\varphi_j/2)(k_j^2+a_j^2)}\prod_{j=m_1+m_2+m_3+1}^{n-m_5} e^{-ia_jk_j/2} . \nonumber\\
\end{eqnarray*}

{\em Step} 6. Our final change of coordinates is the following: For 
$j=m_1+m_2+m_3+1,\ldots ,n-m_5$, we let 
$$
\mat{x_j\\p_j} =
\frac1{\sqrt2}\mat{1&1\\-1&1}\mat{u_j\\v_j}- \mat{a_j/2\\k_j/2}
$$
so that 
\begin{eqnarray*}%
x_j &=& \frac1{\sqrt2}(u_j+v_j) -a_j/2\\
p_j&=& \frac1{\sqrt2} (-u_j+v_j)-k_j/2
\end{eqnarray*}
and 
$$2x_jp_j + a_jp_j+k_jx_j=-u_j^2 +v_j^2-\frac{a_jk_j}2.$$
This leads to 
\begin{eqnarray*}
\lefteqn{
\Tr(R_gT_wA(H+\mu^{\fm})^{-K})}\\ 
&=& 
 (2\pi)^{-n+\frac{m_2+m_3}2} 
\prod_{j=1}^{m_1}  \sqrt{1+i\tg\varphi_j} 
   \prod_{j=1}^{m_1+m_2+m_3}
 e^{\frac i4\ctg(\varphi_j/2)(k_j^2+a_j^2)} \\ 
&&\times 
\int e^{i\phi_6(u',v',u'',u''',v''')} 
\\
&&\times 
\tr_Eq(B'\mat{u'\\v'}-b'_0,\mat{u''\\v''}-b_0'',B'''\mat{u'''\\v'''}-b_0''',x'''',p'''';\mu\Big)%
_{|v''_j= \sin\gvp_j u''_j}  \\
&&\times \, du'dv'du'' du''' dv'''dx''''dp'''',
\end{eqnarray*} 
with $B''$ made up from the matrices $\frac1{\sqrt2}\mat{1&1\\-1&1}$, 
\begin{eqnarray*}%
b'''_0&=& \mat{a_j/2\\k_j/2},\\
u'''&=&(x_{m_1+m_2+m_3+1},\ldots ,x_{n-m_5}),\\
v'''&=&(p_{m_1+m_2+m_3+1},\ldots ,p_{n-m_5}),\\
x''''&=&(x_{n-m_5+1},\ldots, x_n),\\
p''''&=&(p_{n-m_5+1},\ldots, p_n),
\end{eqnarray*}
and 
\begin{eqnarray*}%
\phi_6&=&
-\sum_{j=1}^{m_1} (\lambda^-_ju_j^2 +\lambda^+_jv_j^2)-\sum_{j=m_1+1}^{m_1+m_2}u_j^2    
+\sum_{j=m_1+m_2+1}^{m_1+m_2+m_3}u_j^2   \\        
&&-\sum_{j=m_1+m_2+m_3+1}^{n-m_5}(-u_j^2+v_j^2).
\end{eqnarray*}
This completes the proof of Proposition \ref{prop_trace}.
\end{proof}

?
\subsection*{Proof of Theorem \ref{thm_resolvent}}
To simplify the notation, we  rewrite the expression found
in Proposition \ref{prop_trace} 
in the form 
\begin{eqnarray}\label{eq:simple} 
\Tr(R_gT_w A(H+\mu^\fm)^{-K}) = 
 C_{\rm res} \int_{\R^{n'}} e^{i\phi(\tilde w)}\tilde q(\tilde B\tilde w-\tilde b_0;\mu)\, d\tilde w.
\end{eqnarray}
Here, 
\begin{itemize} 
\item 
$\tilde q$ is $\tr_Eq$ with  $v_{m_1+1} ,\ldots,v_{m_1+m_2+m_3} $, evaluated at 
$v_j =  \sin\gvp_j u_j=\pm u_j$,  so that $n'=2n-m_2-m_3$ integration variables remain.
\item $\tilde B\in \mathrm O({n'})$ and $\tilde  b_0\in \R^{n'}$ incorporate the components of $B$ and $b_0$,
\item 
$\phi(\tilde w) =
\sum_{j=1}^{n'}\tilde \lambda_j \tilde w_j^2 $ is constructed from \eqref{eq_phase}, with  
$\tilde \lambda_j=0$ for $j>n'-2m_5$.
\end{itemize}

\subsubsection*{Step $1$. The  Structure of $\tilde q(\tilde B\tilde w - \tilde b_0; \mu)$}
Let us briefly recall how the right symbol $q$ of $A(H+\mu^\fm)^{-K}$ in \eqref{AHmu} is constructed.  

By Theorem \ref{parametrix} in the appendix on weakly parametric symbol claases, the (left quantized)  Shubin symbol of $(H+\mu^\fm)^{-1}$ has an asymptotic expansion with leading term $(h_\fm + \mu^\fm)^{-1}\in S^{-\fm,0}\cap S^{0,-\fm}$ and lower order terms in $S^{-\fm-j,0}\cap S^{\fm-j,-2\fm}$, $j=1,2,\ldots$, that are linear combinations of the form 
$$(\partial^{\ga_1}_p\partial_x^{\gb_1} h_{\fm-k_1})\ldots 
(\partial^{\ga_r}_p\partial_x^{\gb_r} h_{\fm-k_r})(h_\fm+\mu^\fm)^{-(r+1)}$$
for $r\ge1$, where 
\begin{eqnarray}
\sum_{l=1}^r{k_l} +\sum_{l=1}^r |\alpha_l|+\sum_{l=1}^r|\gb_l| =j\text{ and }\sum_{l=1}^r |\ga_l|=\sum_{l=1}^r|\gb_l|.
\end{eqnarray}
Here we have used the fact that the principal symbol is scalar, so that the powers of $(h_\fm+\mu^\fm)^{-1}$ can be collected. 

Hence the (left) symbol of $(H+\mu^\fm)^{-K} $  has $(h_\fm+\mu^\fm)^{-K}\in S^{-K\fm,0}\cap S^{0,-K\fm}$ as its leading term and lower order terms in $S^{-K\fm-j,0}\cap S^{\fm-j,-(K+1)\fm}$; the expansion is modulo $S^{-\infty,-(K+1)\fm} $. 
The symbol of  $A(H+\mu^\fm)^{-K}$ with  $A=\op(a)\in \Psi^\fa_{\rm cl}$, $a\sim \sum a_{\fa-k}$, has 
$a_\fa(h_\fm+\mu^\fm)^{-K}\in S^{\fa-K\fm,0}\cap S^{\fa,-K\fm}$ as its leading term and  lower order terms in $S^{\fa-K\fm-j,0}\cap S^{\fa-j,-K\fm}$; they are linear combinations of
\begin{eqnarray}\label{4.2.4}
\partial^{\ga_0}_p\partial^{\gb_0}_x a_{\fa-k_0}(\partial^{\ga_1}_p\partial_x^{\gb_1} h_{\fm-k_1})\ldots 
(\partial^{\ga_r}_p\partial_x^{\gb_r} h_{\fm-k_r})(h_\fm+\mu^\fm)^{-K-r},
\end{eqnarray}
where $r\ge 1$ and 
\begin{eqnarray}
\sum_{l=0}^r{k_l} +\sum_{l=0}^r |\alpha_l|+\sum_{l=1}^r|\gb_l| =j.
\end{eqnarray} 
The asymptotic expansion of the right (or $y$-form) symbol is obtained by taking derivatives of the terms in the expansion of the left symbol; hence they have the same structure. Since the symbol classes are preserved under the transition from left to right symbols, the 
remainder is again an element of $S^{-\infty,-K\fm}$. 
  
Noting that the terms in front of the power $(h_\fm+\mu^\fm)$ are all homogeneous 
in $(x,p)$ for $|(x,p)|\ge1$, we obtain part (a) of the following proposition.

\begin{proposition}\label{4.2}
{\rm (a)} For each $N\in \N$ the components of  $q= q(y,p,\mu)$ in \eqref{AHmu}
 can be written as a finite sum of terms of the form 
\begin{eqnarray}\label{4.2.1}%
d(y,p)(h_\fm(y,p) +\mu^\fm)^{-K-\ell} ,
\end{eqnarray}
where $d$ is homogenous in $(y,p)$ of degree $\fd=\fa+\ell\fm-j$ for $j=0,\ldots, N-1$, 
plus a remainder in $S^{-N,-K\fm}\cap S^{-N-K\fm,0}$. 

{\rm (b)} 
Each $d$ is a polynomial in derivatives of the homogeneous components in the symbol expansions of $a$ and $h$. If $\ell=0$ in {\rm(a)}, then $d$ can only be one of the homogeneous components of $a$. 

{\rm (c)} The statement corresponding to (a) is also true for the symbol $\tilde q$. Here,
$d$ is additionally scalar, since we take the fiber trace in $E$. 

{\rm (d)} Also 
$\tilde q(\tilde B\tilde w-\tilde b_0;\mu)$ has an asymptotic expansion into homogeneous components of the 
form \eqref{4.2.1}, with $h_\fm(y,p)+\mu^\fm$ replaced by $h_\fm(\tilde B \tilde w) +\mu^\fm$.  

{\rm (e)} The only possible coefficient functions of $(h_\fm(\tilde B\tilde w)+\mu^\fm)^{-K}$ in this expansion are of the form 
$$\frac{(-\tilde b_0)^\ga}{\ga!}\partial_w^\ga \tr_E a_{\fa-j}(\tilde B \tilde w)$$
for suitable $j\in \N_0$ and a multi-index $\ga$.
\end{proposition} 

It will therefore suffice to consider the two types of terms in Proposition \ref{4.2}(a), i.e. the homogeneous components and the residual terms. 

\begin{proof} We just saw (a); 
(b) is a consequence of \eqref{4.2.4}; 
(c) follows  from (a), and  (d) from (b)  
and Proposition \ref{4.5}(c). This also shows (e). 
 \end{proof}  

\subsubsection*{Step $2$. The Contribution from the Residual Terms}

\begin{lemma} \label{5.5}
Let $r\in S^{-N,-K\fm}$ and $n_0\in \N$ such that $N>2n+n_0$. Then 
\begin{eqnarray}\label{4.1}%
\int e^{i\phi(w)} r(w;\mu)\, dw =\sum_{k=0}^{n_0-1} c_k \mu^{-K\fm-k} + 
O(\mu^{-K\fm-n_0})  
\end{eqnarray}
with suitable coefficients $c_k$.  
In the case of Proposition \ref{4.2}, the  $c_k$ are zero unless $k$ is an integer multiple of $\fm$.
\end{lemma}

\begin{proof} 
As in \eqref{exp_infty} in the Appendix, 
define $r_{(-K\fm,k)}(w)  =\frac1{k!} \partial^k_t(t^{-K\fm}r(w;\frac1t))_{|t=0}\in S^{-N+k}$.
According to Theorem \ref{GrubbS2}, 
\begin{eqnarray*}%
s(w;\mu) := r(w;\mu) - \sum_{k=0}^{n_0-1}\mu^{-K\fm-k}r_{(-K\fm,k)}(w) \in S^{-N+n_0,-K\fm-n_0}.
\end{eqnarray*}  
In particular, $\mu^{K\fm+n_0} s(w,\mu)\in S^{-N+n_0}$, uniformly in $\mu$. 
In view of 
the fact that $-N+n_0<-2n$ by assumption, we see that
$$\mu^{K\fm+n_0} \int_{\R^{n'}} e^{i\phi(w)} s(w;\mu)\, dw = O(1).$$
Since $r_{(-K\fm,k)} \in S^{-N+k}$ and $-N+k<-2n$ for $k<n_0$, we obtain \eqref{4.1}. 

In the case at hand, the symbol is a smooth function of $-\gl=\mu^\fm$. 
Therefore  the Taylor coefficients $r_{-K\fm-k}$ are zero unless $k$ is a multiple of $\fm$. This was already pointed out by Grubb and  Seeley \cite[p.501]{GS1}.
\end{proof}

\subsubsection*{Step $3$. Contributions from the symbol expansion} 
Next we focus on the homogeneous terms in the asymptotic expansion of  $\tilde q(\tilde B\tilde w-\tilde b_0;\mu)$. 
 According to Proposition \ref{4.2} and Lemma \ref{4.5}, it suffices to consider terms of the form 
\begin{eqnarray}\label{simplified_integrand}%
\int_{\R^{n'}} e^{i\phi(w)}
d(w) (h_\fm(\tilde B w) + \mu^\fm)^{-K-\ell} \, dw,
\end{eqnarray}
where $\ell\in \N_0$ and $d$ is a Shubin symbol in $w$, which is homogeneous of  some
degree $\fd= \fa+ \ell \fm-j$, $j\in \N_0$,  for $|w|\ge1$.  
 We shall write 
$$\tilde h_\fm = h_\fm\circ \tilde B.$$
Since $\tilde B$ is  orthogonal,  $\tilde h_\fm$ satisfies the assumptions imposed on $h_\fm$. 

Similarly as Grubb and Seeley in their proof of \cite[Theorem 2.1]{GS1}, we shall split the integral over $\R^{n'}$ into the integrals 
over $\{|w|\le 1\}$, $\{1 \le|w|\le |\mu|\}$, and $\{|w|\ge |\mu|\}$. 

\subsubsection*{Step $4$. The integral over $\{|w|\le 1\}$}

\begin{lemma} \label{4.7}
The integral 
$$\int_{|w|\le 1}e^{i\phi(w)} d(w) (\tilde h_\fm (w)+\mu^\fm)^{-K-\ell}\, dw$$
has an expansion into integer powers $\mu^{-(K+\ell)\fm}, \mu^{-(K+\ell-1)\fm}, \ldots$. 
\end{lemma} 

\begin{proof} 
We write 
$$(\tilde h(w)+\mu^\fm)^{-K-\ell}\ = \mu^{-(K+\ell)\fm}(\tilde h_\fm(w)/\mu^\fm +1)^{-K-\ell}.$$
To the last factor, we may apply the binomial series expansion. Since it converges uniformly for large $|\mu|$
we obtain the assertion. 
\end{proof}

\subsubsection*{Step $5$. The integral over $\{|w|\ge |\mu|\}$. Preparations}

Letting $v=w/|\mu|$ and using the facts that $\phi$ is 2-homogeneous and  
$d$ is homogeneous of degree $\fd$
\begin{eqnarray}\label{4.7.2}
\lefteqn{
\int_{|w|\ge |\mu|} e^{i\phi(w)} d(w) (\tilde  h_\fm (w)+\mu^\fm)^{-K-\ell} \, dw}\nonumber\\
&=&|\mu|^{\fd-\fm(K+\ell)+n'} \int_{|v|\ge 1} e^{i\phi(v)|\mu|^2} d(v)  
(\tilde h_\fm(v) + (\mu/|\mu|)^\fm)^{-K-\ell} \, dv.
\end{eqnarray}
Introducing polar coordinates $v=r\theta$, $r=|v|$, and letting 
$$\go=\mu/|\mu|,$$ 
this integral can be rewritten
\begin{eqnarray}\lefteqn{
\hspace*{-1cm} |\mu|^{\fd-\fm(K+\ell)+n'} 
\int_1^\infty \int_{\mathbb S^{n'-1}}
e^{i\phi(\gt)r^2|\mu|^2} d(r\gt) (\tilde h_\fm(r\gt) + \go^\fm)^{-K-\ell}\, dS(\theta)
r^{n'-1} dr}
\nonumber\\
&=&|\mu|^{\fd-\fm(K+\ell)+n'} \int_1^\infty I(\hbar, r, \go)_{|\hbar=(|\mu|^2r^2)^{-1}}
r^{\fd+n'-1}\, dr
\label{4.7.3}
\end{eqnarray}
with 
\begin{eqnarray}%
I(\hbar, r,\go) = \int_{\mathbb S^{n'-1}} e^{i\phi(\gt)/\hbar}d(\gt) 
(\tilde h_\fm(r\gt)+\go^\fm)^{-K-\ell}  \, dS(\gt).\label{4.7.4}
\end{eqnarray}

We will obtain an expansion for this integral from the stationary phase approximation, see e.g. \cite[Theorem 3.16]{Z} or \cite{Fed1}. 
Since $\phi(\theta)=
\sum_{j=1}^{n'}   \tilde \lambda_j\theta^2_j$, the stationary 
points are  the unit length eigenvectors of the matrix 
$\text{\rm diag}(\tilde \lambda_1, \ldots, \tilde  \lambda_{n'})$. \medskip

{\bf Notation.} We denote by $\{\mu_l\}$ the set of different $\tilde \lambda_j$ and by $\kappa_l$ the multiplicity of $\mu_l$. 
The set of stationary points therefore
is the disjoint union of the spheres $\mathbb S^{\gk_l-1}$. 
A central role will be played by the eigenvalue zero.  
For definiteness we assume that this is $\mu_1$, i.e., $\mu_1=0$ 
with multiplicity $\gk_1=2m_5$. \medskip

\begin{lemma}\label{4.8}
For $N\in \N$, stationary phase approximation gives us
\begin{eqnarray*}\lefteqn{
I(\hbar, r,\go) }\\
&=& \sum_{J=0}^{N-1}
\sum_le^{\frac i\hbar \mu_l}
\hbar^{J+\frac{n'-\kappa_l}2} 
\int_{\mathbb S^{\kappa_l-1}}A_{2J}(\theta', \partial_{\theta'})\left( d(\theta) 
(\tilde h_\fm(r\gt)+\go^\fm)^{-K-\ell}\right)\, dS\\
&&+ O\left(\sum_l\hbar^{N+\frac{n'-\gk_l}2} \sum_{|\alpha|\le 2N+n'-\kappa_l+1} \sup |\partial^\alpha_{\theta'}\big(d(\theta)(\tilde h_\fm(r\gt)+\go^\fm)^{-K-\ell}\big)|\right)\label{7.2} 
\end{eqnarray*}
with differential operators $A_{2J}(\gt',\partial_{\gt'})$ of order $2J$ in the variables $\theta'$ transverse to the fixed point set. 
\end{lemma}

\begin{lemma} \label{4.10} 
{\rm (a)} For every multi-index $\ga$, 
$\partial^\ga_{\gt'} (d(\gt) (\tilde h_\fm(r\gt)+\go^\fm)^{-K-\ell})$ is a linear combination 
of terms of the form 
$$r^{\fm m}\tilde d(\gt) (\tilde h_\fm(r\gt)+\go^\fm)^{-K-\ell -m},$$ 
where $ m\le|\ga|$ 
and $\tilde d$ is smooth on $\mathbb S^{n'-1}$. 
In particular,  
$$\partial^\ga_{\gt'} (d(\gt) (\tilde h_\fm(r\gt)+\go^\fm)^{-K-\ell})= O(r^{-\fm(K+\ell)}),\quad r\to \infty.$$

{\rm (b)}   $\partial^j_{r}  (\tilde h_\fm(r\gt)+\go^\fm)^{-K-\ell}$, $j\in \N$,
is a linear combination of terms 
$$r^{\fm m-j}\tilde d(\gt) (\tilde h_\fm(r\gt)+\go^\fm)^{-K-\ell-m}$$
with $1\le m\le j$ and $\tilde d$ smooth. 
In particular,  
$$\partial^j_{r}  (\tilde h_\fm(r\gt)+\go^\fm)^{-K-\ell}= O(r^{-\fm(K+\ell)-j}), \quad r\to \infty.$$
\end{lemma} 

\begin{proof}
(a) The derivative is a linear combination of terms of the form 
$$\partial_{\gt'}^{\ga_1}d(\gt)\  r^{|\ga_2|} \prod_{k=1}^m(\partial^{\gb_k}_{\gt'}\tilde h_\fm(r\gt))\ (\tilde h_\fm(r\gt)+\go^\fm)^{-K-\ell-m},$$ 
where $\ga_1+\ga_2=\ga$, $\sum\gb_k = \ga_2$ and $m\le|\ga_2|$. This proves the first part of the claim. In view of the fact that $r\sim \skp{r\gt}$ for large $r$ and $|\gt|=1$, we also obtain the estimate. 

(b)  follows since $\partial_r^j (\tilde h_\fm(r\gt)+\go)^{-K-\ell}$ is a linear combination of terms 
$$\prod_{k=1}^m \partial_r^{j_k}(\tilde h_\fm(r\gt) ) (\tilde h_\fm(r\gt)  +\go^\fm)^{-K-\ell-m}$$ 
with $m\le j$ and $\sum{j_k}= j$.
\end{proof}

\begin{corollary}\label{4.11}
The contribution of the remainder term in Lemma {\rm \ref{4.8}}
to 
the integral in~\eqref{4.7.3} is $O(|\mu|^{-2N-n'+\max\{\gk_l\}})$ 
provided $2N- \max\{\gk_l\}+\fm(K+\ell)-\fd>0$. 
\end{corollary}

\begin{proof}
By Lemma \ref{4.10}(a), the $\theta'$-derivatives in the integrand of the remainder term are all $O(r^{-\fm(K+\ell)})$. Hence the integral over $r$
can be estimated by 
\begin{eqnarray*}%
\int_1^\infty (r^2|\mu|^2)^{-N-(n'-\max\{\gk_l\})/2} r^{-\fm(K+\ell)}r^{\fd+n'-1}\, dr
=O(|\mu|^{-2N-n'+\max\{\gk_l\}}). 
\end{eqnarray*}
\end{proof} 

\subsubsection*{Step $6$. The terms from the stationary phase expansion}
 Recalling that $\hbar^{-1} = |\mu|^2r^2$, the contribution of the sum in Lemma \ref{4.8} to the integral \eqref{4.7.3} is
\begin{eqnarray}\label{4.12.1}\lefteqn{%
\sum_{J=0}^{N-1}\int_1^\infty (|\mu|^2r^2)^{-J-\frac{n'-\kappa_l}2}
e^{ir^2|\mu|^2\mu_l}}\\
&&\times\int_{\mathbb S^{\kappa_l-1}}A_{2J}(\theta', \partial_{\theta'})\left( d(\theta) 
(\tilde h_\fm(r\gt)+\go^\fm)^{-K-\ell}\right)\, dS \,r^{\fd+n'-1}\, dr. \nonumber
\end{eqnarray} 

\begin{lemma}\label{4.12}
For $J=0,\ldots,N-1$, the integral 
\begin{eqnarray}\label{4.12.2}\lefteqn{
\int_1^\infty e^{ir^2|\mu|^2\mu_l} r^{-2J+\kappa_l+\fd-1}}\\
&&\times \int_{\mathbb S^{\kappa_l-1}}A_{2J}(\theta', \partial_{\theta'})\left( d(\theta) 
(\tilde h_\fm(r\gt)+\go^\fm)^{-K-\ell}\right)\, dS dr\nonumber
\end{eqnarray}
has an asymptotic expansion into terms of the form $e^{i|\mu|^2\mu_l} |\mu|^{-k}g_{kl}(\go)$
for suitable $k\in \N_0$ and smooth functions $g_{kl}$.  
\end{lemma} 

These terms seem not to be of the desired form, in particular, when  $\mu_l\not=0$. We will address this question in Lemma \ref{correct}, below. 

\begin{proof}
The integral converges in view of Lemma \ref{4.10}, since $\fd = \fa+\ell\fm -j$ for some $j\in \N_0$ and  $K\fm-\fa-\max \gk_l>2n-n'\ge 0$ by the assumption in Theorem \ref{thm_resolvent}. 
If $\mu_l=0$, then the integral is independent of $\mu$ and the assertion is true. 

So let us focus on the case  $\mu_l\not=0$. 
Applying Lemma \ref{4.10}(a), we have to consider integrals 
$$\int_1^\infty e^{ir^2|\mu|^2\mu_l} r^{-2J+\kappa_l+\fd +\fm m-1}
\int_{\mathbb S^{\kappa_l-1}}\tilde d(\gt) (\tilde h_\fm(r\gt)+\go^\fm)^{-K-\ell-m}\, dS\,dr $$
with smooth functions $\tilde d$ on the sphere. 

Since $(\partial_r e^{ir^2|\mu|^2\mu_l})r^{M-1} = 2i |\mu|^2\mu_l  \, e^{ir^2|\mu|^2\mu_l}r^{M}$,
$M\in \Z$, integration by parts with $M= -2J+\gk_l+\fd+\fm m -1$ shows that  
\begin{eqnarray*}\lefteqn{
\int_1^\infty e^{ir^2|\mu|^2\mu_l} r^{M}
\int_{\mathbb S^{\kappa_l-1}}\tilde d(\gt) (\tilde h_\fm(r\gt)+\go^\fm)^{-K-\ell -m}\, dS\,dr }\\
&=& \frac{1}{2i|\mu|^2\mu_l}\left(\left[ e^{ir^2|\mu|^2\mu_l}r^{M-1}
\int_{\mathbb S^{\kappa_l-1}}
\tilde d(\gt) (\tilde h_\fm(r\gt)+\go^\fm)^{-K-\ell -m}\, dS \right]_1^\infty \right.\\
&&-\int_1^\infty e^{ir^2|\mu|^2\mu_l} (M-1) r^{M-2}
\int_{\mathbb S^{\kappa_l-1}}\tilde d(\gt) (\tilde h_\fm(r\gt)+\go^\fm)^{-K-\ell -m}\, dS\,dr 
\\
&&- \left.\int_1^\infty e^{ir^2|\mu|^2\mu_l} r^{M-1}
\int_{\mathbb S^{\kappa_l-1}}\tilde d(\gt) \partial_r(\tilde h_\fm(r\gt)+\go^\fm)^{-K-\ell -m}\, dS\,dr\right)
\end{eqnarray*}  
\begin{eqnarray*}
&=&- \frac{e^{i|\mu|^2\mu_l}}{2i|\mu|^2\mu_l}  \int_{\mathbb S^{\kappa_l-1}}
\tilde d(\gt) (\tilde h_\fm(\gt)+\go^\fm)^{-K-\ell -m}\, dS\\
&&-\frac{1}{2i|\mu|^2\mu_l}\left(\int_1^\infty e^{ir^2|\mu|^2\mu_l} (M-1) r^{M-2}
\int_{\mathbb S^{\kappa_l-1}}\tilde d(\gt) (\tilde h_\fm(r\gt)+\go^\fm)^{-K-\ell -m}\, dS\,dr\right. 
\\
&&- \left.\int_1^\infty e^{ir^2|\mu|^2\mu_l} r^{M-1}
\int_{\mathbb S^{\kappa_l-1}}\tilde d(\gt) \partial_r(\tilde h_\fm(r\gt)+\go^\fm)^{-K-\ell -m}\, dS\,dr\right).
\end{eqnarray*}  
Iteration together with Lemma \ref{4.10}  and the fact that all integrals are uniformly bounded in $\mu$ then shows the assertion. 
\end{proof}

\subsubsection*{Step $7$. The integral over $\{1\le |w| \le |\mu|\}$}
Using polar coordinates 
\begin{eqnarray}
\lefteqn{
\int_{1\le |w|\le |\mu|} e^{i\phi(w)} d(w) (\tilde h_\fm(w) +\mu^\fm)^{-K-\ell} \, dw }
\nonumber\\
&=&\int_{1}^{|\mu|} \int_{\mathbb S^{n'-1}}e^{ir^2\phi(\gt)} d(r\gt) (\tilde h_\fm(r\gt) +\mu^\fm)^{-K-\ell} r^{n'-1}\, dSdr
\nonumber
\end{eqnarray}

In order to obtain an expansion, we  recall the geometric sum 
$$\sum_{k=0}^{N-1}x^k = \frac{1-x^N}{1-x}.$$ 
Taking the difference of the $(K+\ell)$-th power of both sides and adding $(1-x)^{-(K+\ell)}$ we
deduce that 
$$(1-x)^{-K-\ell}= \left( \sum_{k=0}^{N-1}x^k \right)^{K+\ell}-\frac{(1-x^N)^{K+\ell}-1}{(1-x)^{K+\ell}}.
$$
Hence \begin{eqnarray}
\lefteqn{
(\tilde h_\fm+\mu^\fm)^{-K-\ell} = \mu^{-\fm(K+\ell)}\big(1+\frac{\tilde h_\fm}{\mu^\fm}\big)^{-K-\ell} }\nonumber\\
&=&\mu^{-\fm(K+\ell)}
\bigg(\Big(\sum_{k=0}^{N-1} \Big(-\frac{\tilde h_\fm}{\mu^\fm}\Big)^k\Big)^{K+\ell}   
-\frac{\left(1 -\big(-\frac{\tilde h_\fm}{\mu^\fm}\big)^N\right)^{K+\ell}-1}{(1+\frac{\tilde h_\fm}{\mu^\fm})^{K+\ell}}
\bigg).  
\label{4.12.3}
\end{eqnarray}

Let us first consider the contribution from the terms in the first summand on the right hand side of \eqref{4.12.3}. We have to treat expressions of the form 
\begin{eqnarray}\lefteqn{
\int_{1}^{|\mu|} \int_{\mathbb S^{n'-1}}
e^{ir^2\phi(\gt)} d(r\gt) (\tilde h_\fm(r\gt)/\mu^{\fm})^k  r^{n'-1}\, dSdr  \nonumber}\\
&=&
\mu^{-\fm k} 
\int_{1}^{|\mu|} \int_{\mathbb S^{n'-1}}e^{ir^2\phi(\gt)}d(\gt) \tilde h_\fm(\gt)^k \, dS
\, r^{\fd+\fm k+n'-1}\, dr.
\label{4.12.4}
\end{eqnarray}

\begin{lemma}\label{4.16} 
 For $\mu_l\not=0$, expression \eqref{4.12.4} has an asymptotic expansion into linear combinations of terms 
$e^{i|\mu|^2\mu_l} |\mu|^{\fd - 2J+\gk_l-1}$, $J=0,1,\ldots$. 

For $\mu_1=0$ we additionally obtain logarithmic terms.
Taking into account the factor $\mu^{-(K+\ell)\fm}$ in \eqref{4.12.3}, the logarithmic terms are of the form 
\begin{eqnarray}\label{coefficients}
C_{j}\mu^{-(K+j)\fm} \ln\mu,\quad  j=0,1,\ldots 
\end{eqnarray}
with suitable coefficients $C_j$.   
As the powers in \eqref{coefficients} are integer multiples of $\fm$, we see that  $\mu^{-(K+j)\fm}= (-\gl)^{-K-j}$ 
is an integer power of $(-\gl)$. 
Moreover, the coefficient  $C_0$ of the leading logarithmic term $\mu^{-K\fm} \ln\mu$ is independent of the auxiliary operator $H$.
\end{lemma} 

Before giving the proof, let us note that this yields: 

{\em Proof of Theorem {\rm\ref{thm_resolvent}(e)}.} 
In view of the fact that $C_0$ in \eqref{coefficients} is independent of $H$ and 
$$\ln(-\gl) = \ln(\mu^\fm) = \fm \ln \mu = \ord H \ln\mu, $$
 the product $\fm c_0'$ in Theorem {\rm \ref{thm_resolvent}}(e) is independent of $H$ as asserted. 
\hfill$\Box$\smallskip

{\em Proof of Lemma} \ref{4.16}.
We use the stationary phase approximation with $\hbar = r^{-2}$ 
for the integral over the sphere. This results in terms 
\begin{eqnarray}
\mu^{-\fm k} c_{l,k,J}
\sum_l \sum_{J=0}^{N-1} 
\int_{1}^{|\mu|}e^{ir^2\mu_l}
\, r^{\fd+\fm k-2J+\gk_l-1}\, dr\label{4.14.1}
\end{eqnarray}
with 
\begin{eqnarray}
\label{4.14.1a}
c_{l,k,J} = \int_{\mathbb S^{\gk_l-1}}A_{2J,l}(\theta',\partial_{\gt'})(d(\gt)\tilde h_\fm(\gt)^k)\, dS
\end{eqnarray}
and a remainder term of the form 
\begin{eqnarray}\label{4.14.2}%
\mu^{-\fm k} \int_1^{|\mu|} F(r)\, dr
\end{eqnarray}
with a function $F$, independent of $\mu$, satisfying $F(r) = O(r^{\fd+\fm k-2N+\gk_l-1})$. 
For sufficiently large $N$ we  write
$$\mu^{-\fm k}\int_1^{|\mu|} F(r)\, dr = 
\mu^{-\fm k}\int_1^\infty F(r)\, dr  - \mu^{-\fm k}\int_{|\mu|}^\infty F(r)\, dr. $$   
and   estimate the second term on the right hand side by
\begin{eqnarray*}%
\mu^{-\fm k} C \int_{|\mu|}^\infty r^{\fd+\fm k-2N+\gk_l-1}\, dr=O( |\mu|^{\fd+\gk_l-2N}).
\end{eqnarray*}

Now for the terms in \eqref{4.14.1}. 

(i) For $l =1$ (and hence $\mu_1=0$) we obtain 
\begin{eqnarray}%
\frac{\mu^{-\fm k} c_{1,k,J}}{\fd+\fm k-2J+\gk_1} (|\mu|^{\fd+\fm k-2J+\gk_1}-1),&&\text{ if } \fd+\fm k-2J+\gk_1\not=0\text{ and} \nonumber\\ 
\mu^{-\fm k} c_{1,k,J} \ln |\mu|,&&\text{ if } \fd+\fm k-2J+\gk_1=0.
\label{log-term}
\end{eqnarray}
We note that $\ln \mu  = \ln(|\mu|\go) = \ln|\mu| + \ln\go$, so that  replacing $\mu^{-k\fm}\ln|\mu|$ by $\mu^{-k\fm} \ln\mu$ causes an error of the form $\mu^{-k\fm}\ln\go$ that will be treated in Lemma \ref{correct}, below. 
Taking into account the  factor $\mu^{-(K+\ell)\fm}$ in \eqref{4.12.3} we see that the powers of $\mu$ associated with the logarithmic term are multiples of $\fm$. 
Moreover, the leading order contribution  is $c_{1,0,J}\mu^{-K\fm} \ln\mu$:
From \eqref{4.12.3} we deduce immediately that it will be 
$o(|\mu|^{-K\fm})$ unless $\ell=k=0$.
By Proposition \ref{4.2}(b)  the only terms in the asymptotic expansion of $q$ of the form $d(y,p)(\tilde h_\fm(y ,p)+\mu^\fm)^{-K}$ with $\ell=0$ and $d$ homogenous for $|w|\ge1$ are those, where  $d$ only depends on the homogeneous components in the symbol expansion of $A$ and its derivatives. 
In particular, $d$ and hence the coefficient of $\mu^{-K\fm} \ln \mu$ are independent of the  operator $H$. 

(ii) For $\mu_l\not=0$, write  
$M=  \fd+\fm k-2J+\gk_l-1$ and integrate by parts:
\begin{eqnarray}
\lefteqn{ \int_1^{|\mu|} e^{ir^2\mu_l} r^M\, dr\label{4.16.7}
=\frac1{2i\mu_l} \int_1^{|\mu|} \partial_r(e^{ir^2\mu_l} )r^{M-1} \, dr}\\
&=&\frac1{2i\mu_l}\left( \left[ e^{ir^2\mu_l} r^{M-1} \right]_1^{|\mu|}-  \int_1^{|\mu|} e^{ir^2\mu_l} (M-1)r^{M-2} \, dr\right)\nonumber\\
&=&\frac1{2i\mu_l}(e^{i|\mu|^2\mu_l}|\mu|^{M-1} -e^{i\mu_l} )  -  
\frac1{2i\mu_l}  \int_1^{|\mu|} e^{ir^2\mu_l} (M-1)r^{M-2} \, dr.\nonumber
\end{eqnarray}

We iterate this finitely many times until we have reached an integral of the form \eqref{4.16.7} with $M\le -2$. Then we write 
\begin{eqnarray*}%
\int_1^{|\mu|} e^{ir^2\mu_l }r^{M} \, dr = \int_1^\infty e^{ir^2\mu_l}r^{M} \, dr 
- \int_{|\mu|}^\infty e^{ir^2\mu_l} r^{M} \, dr .
\end{eqnarray*} 
The first term is a constant independent of $\mu$. To the second we can apply the above integration by parts technique in order to obtain an expansion into terms of the form $c_{l,j}\, e^{i|\mu|^2\mu_l}|\mu|^{M-j}$, for suitable constants $c_{l,j}$, $j\in \N$, plus a remainder that is  $O(|\mu|^{-N})$ for arbitrary $N$  by the same argument as for \eqref{4.14.2}. 
\hfill$\Box$\smallskip

Next we analyze the contribution from the remainder term on the right hand side of \eqref{4.12.3}. We will have to consider terms of the form 
\begin{eqnarray}\label{4.14.5}\lefteqn{
\mu^{-\fm(K+\ell)}\ \int_{1}^{|\mu|} \int_{\mathbb S^{n'-1}}e^{ir^2\phi(\gt)} 
d(r\gt) 
({\tilde h_\fm(r\gt)}{\mu^{-\fm}} )^{mN}}\nonumber\\
\nonumber
&&\ \times (1+{\tilde h_\fm(r\gt)}{\mu^{-\fm}})^{-K-\ell} r^{n'-1}\, dSdr\\
&=&
\mu^{-\fm(mN+K+\ell)}\ \int_{1}^{|\mu|} \int_{\mathbb S^{n'-1}}e^{ir^2\phi(\gt)}
d(r\gt) 
\tilde h_\fm(r\gt)^{mN}
\label{4.14.5a}
\\
&&\ \times \left(1+{\tilde h_\fm(r\gt)}{\mu^{-\fm}}\right)^{-K-\ell} r^{n'-1}\, dSdr
\nonumber
\end{eqnarray}
for $m=1,2,\ldots, K+\ell$. 
We extend $d$ and $\tilde h_\fm$ by homogeneity to functions $d^{\sf \,hom}$ and 
$\tilde h_\fm^{\sf hom}$ on $\R^{n'}\setminus \{0\}$ and rewrite the  integral in 
\eqref{4.14.5a}
as 
\begin{eqnarray}
\lefteqn{
\int_{0}^{|\mu|} \int_{\mathbb S^{n'-1}}e^{ir^2\phi(\gt)}
d^{\sf \,hom}(r\gt) 
\tilde h^{\sf hom}_\fm(r\gt)^{mN}}\label{4.14.8}
\\
&&\left(1+{\tilde h^{\sf hom}_\fm(r\gt)}{\mu^{-\fm}}\right)^{-K-\ell} r^{n'-1}\, dSdr
\nonumber
\\
&&
\label{4.14.10} 
- \int_{0}^{1} \int_{\mathbb S^{n'-1}}e^{ir^2\phi(\gt)}
d^{\sf \,hom}(r\gt) 
\tilde h^{\sf hom}_\fm(r\gt)^{mN}
\\&&\ \times 
\left(1+{\tilde h^{\sf hom}_\fm(r\gt)}{\mu^{-\fm}}\right)^{-K-\ell} r^{n'-1}\, dSdr.
\nonumber
\end{eqnarray}
The integral  \eqref{4.14.10} defines
an analytic function of $\mu^{-\fm}$, provided $N$ is chosen suitably large. 
It therefore has an expansion into powers $\mu^{-j\fm}$, $j=0,1,\ldots$.  

Using the homogeneity and making the change of variables $r=s|\mu|$, we rewrite the integral \eqref{4.14.8} as
\begin{eqnarray}\lefteqn{\nonumber
|\mu|^{\fd+\fm mN+n'}\ \int_0^1\int_{\mathbb S^{n'-1}} e^{i\phi(\gt)|\mu|^2s^2} d^{\sf hom}(\gt) 
\tilde h_\fm^{\sf hom}(\gt)^{mN} }\\
&&\times \left(1+{\tilde h_\fm^{\sf hom} (s\gt)}{\go^{-\fm}}\right)^{-K-\ell}dS\ 
s^{\fd+\fm mN+n'-1} \, ds.\label{4.14.11}
\end{eqnarray}

\begin{lemma}\label{4.18} 
The expression \eqref{4.14.11} has an asymptotic expansion into terms of the form 
$$ e^{i|\mu|^2\mu_l} |\mu|^{\fd+\fm m N+\gk_l-j}g_{l,j}(\go) $$
for smooth functions $g_{l,j}$, $j\in \N_0$. 
\end{lemma} 

\begin{proof} 
The integral in \eqref{4.14.8} can be written
\begin{eqnarray*}
|\mu|^{\fd+\fm m N +n'}\int_0^1 I(\hbar;s,\go)_{|\hbar=|\mu|^{-2}s^{-2}}  \ 
s^{\fd+\fm mN+n'-1} \, ds,
\end{eqnarray*}
where 
\begin{eqnarray*}%
I(\hbar;s,\go) = \int_{\mathbb S^{n'-1}} e^{\frac i\hbar\phi(\gt)} d^{\sf hom}(\gt) 
\tilde h_\fm^{\sf hom}(\gt)^{mN} 
\left(1+{\tilde h_\fm^{\sf hom} (s\gt)}{\go^{-\fm}}\right)^{-K-\ell}dS.
\end{eqnarray*}

Similarly as in the analysis of \eqref{4.7.3}/\eqref{4.7.4}, 
we apply the stationary phase approximation. 
With the notation introduced after \eqref{4.7.4} we obtain, as in Lemma \ref{4.8}, 
for $\tilde N\in \N$, 
\begin{eqnarray*}\lefteqn{
I(\hbar; s,\go) = \sum_{J=0}^{\tilde N-1}
\hbar^{J+\frac{n'-\kappa_l}2} 
\sum_l e^{\frac i\hbar \mu_l}}\\
&&\times \int_{\mathbb S^{\kappa_l-1}}A_{2J}(\theta', \partial_{\theta'})
\left( d^{\sf hom\,}(\theta) h^{\sf hom}_\fm(\gt)^{mN} 
(1+{h^{\sf hom}_\fm(s\gt)}{\go^{-\fm}})^{-K-\ell}\right)
\, dS\\
&+& O\Big(\sum_l\hbar^{\tilde N+\frac{n'-\gk_l}2} \sum_{|\alpha|\le 2\tilde N+n'-\kappa_l+1} 
\sup \left|\partial^\alpha_{\theta'}\left(d^{\sf hom\,}(\theta)h^{\sf hom}_\fm(r\gt)^{mN}
\right.\right.\\
&&\left.\left.\left.\times(1+{h^{\sf hom}_\fm(s\gt)}{\go^{-\fm}})^{-K-\ell}\right)\right|\right)
\end{eqnarray*}
with differential operators $A_{2J}$ of order $2J$ in the variables $\theta'$ transversal to the fixed point set.

Let us first study the terms under the summation.  
Evaluating $\hbar$ at $s^{-2} |\mu|^{-2}$ yields expressions  
\begin{eqnarray}
\lefteqn{\sum_l  |\mu|^{-2J-n'+\kappa_l}
\int_0^1 e^{ is^2|\mu|^2 \mu_l}
\nonumber}\\
&&\times \int_{\mathbb S^{\kappa_l-1}}A_{2J}(\theta', \partial_{\theta'})
\left( d^{\sf hom\,}(\theta) h^{\sf hom}_\fm(\gt)^{mN} 
(1+{h^{\sf hom}_\fm(s\gt)}{\go^{-\fm}})^{-K-\ell}\right)
\, dS \label{4.18.2}\\
&&\times s^{-2J+\kappa_l+\fd+\fm mN-1} ds.\nonumber
\end{eqnarray}

We note that, for $M\in \Z$, and $0< s\le 1$
\begin{eqnarray*}\lefteqn{
\partial_{\gt'}(1+h^{\sf hom}_\fm(s{\gt}) \go^{-\fm})^{M} }\\
&=&  sM (1+h^{\sf hom}_\fm(s\gt) \go^{-\fm})^{M-1} 
\go^{-\fm} (\partial_{\xi} \tilde h_\fm^{\sf hom}) (s{\gt}) = O(s^\fm) ;
\end{eqnarray*}
by induction, also the higher $\theta'$-derivatives are $O(s^\fm)$. Similarly,  
\begin{eqnarray*}%
\lefteqn{
\partial_{s}(1+h^{\sf hom}_\fm(s{\gt}) \go^{-\fm})^{M} }\\
&=&  M (1+h^{\sf hom}_\fm(s{\gt}) \go^{-\fm})^{M-1} 
\go^{-\fm} (\partial_{\xi} \tilde h_\fm^{\sf hom}) (s{\gt}) \gt= O(s^{\fm-1}) 
\end{eqnarray*}
and the $j$-th $s$-derivative is $O(s^{\fm-j})$. 
In particular, the integral over the sphere in \eqref{4.18.2} 
is uniformly bounded in $s$ for $0<s\le1$.  
 
We choose $\tilde N$ so that $-2\tilde N +\kappa_l+\fd+\fm N>0$, so that the integral in $s$ over $[0,1] $ exists.  
For  $J=0, \ldots, \tilde N-1$  we may then 
apply the  integration by parts technique we used in the proof of Lemma \ref{4.12}. 
In view of the estimates on $s$-derivative of $(1+h^{\sf hom}_\fm(s{\gt}) \go^{-\fm})^{M}$ above, we obtain, for each $\mu_l\not=0$, 
an expansion into terms $e^{i|\mu|^2\mu_l} |\mu|^{\gk_l-n'-j}g_{jl}(\go)$ for $j=2J,\ldots ,2\tilde N-2$ and smooth functions  $g_{jl}$, plus a remainder which is $O(|\mu|^{\gk_l-n'-2\tilde N})$. 

Finally we consider the remainder term in the stationary phase approximation. 
Evaluating at $\hbar = s^{-2}|\mu|^{-2}$ and using the above estimates on the
$\theta'$-derivatives of the powers of $(1+h^{\sf hom}_\fm(s{\gt}) \go^{-\fm})$, 
we see that the remainder is $O(|\mu|^{\kappa_l-n'-2\tilde N})$.

Since $N$ could be taken arbitrarily large, also $\tilde N$ can be made arbitrarily large, and so we obtain the desired expansion. 
\end{proof}

\subsubsection*{Step $8$. Correcting the expansion} 
As pointed out above, our analysis produces 
a number of terms that  are not of the form stated in Theorem \ref{thm_resolvent}, see Lemmas
\ref{4.12}, \ref{4.16}, \ref{4.18}.
As we will see in Lemma \ref{correct}, they necessarily add up to zero. This is a phenomenon already observed by Grubb and Seeley,  see \cite[Lemma 2.3]{GS1}. \medskip

\begin{lemma}\label{correct}
The expansion of \eqref{eq:simple} is only in powers of $\mu$ and $\ln \mu$ as stated in  
Theorem {\rm \ref{thm_resolvent}}.
\end{lemma} 

\begin{proof}
Consider the leading order 
term and write 
$$\Tr(R_gT_wA(H+\mu^\fm)^{-K})= 
(\mu^d\ln^k \mu) \Big(\sum e^{i\mu_l|\mu|^2}g_l(\omega)
+ h(\omega) \Big) + o(|\mu|^d\ln^k |\mu|)$$
as $\mu\to \infty$ in $S$ 
 with $k\in \{0,1\}$. Clearly, we may restrict the summation to the $\mu_l\not=0$. 

Equivalently we have 
\begin{eqnarray}\label{10.1}
\lefteqn{\hspace*{-5cm} \sum e^{i\mu_l|\mu|^2}g_l(\omega)+ h(\omega)
= (\mu^{-d}\ln^{-k} \mu)\Tr(R_gT_wA(H+\mu^\fm)^{-K}) + o(1).}
\end{eqnarray}
Given a point $z=re^{i\varphi}$, $r>0$, 
$|\varphi|<\pi$,
and $r_1<r<r_2$ close to $r$, 
$\varphi_1<\varphi<\varphi_2$ close to $\varphi$ we next define a contour $\gamma$ as follows (see Fig.~1): 
From $r_1 e^{i\varphi_1}$ to $r_2e^{i\varphi_1}$ along the ray $\arg z = \varphi_1$, from 
$r_2e^{i\varphi_1}$ to $r_2e^{i\varphi_2}$ counterclockwise along the circle	 $|z|=r_2$, from 
$r_2e^{i\varphi_2}$ to $r_1e^{i\varphi_2}$ along the ray $\arg z = \varphi_2$,  and from 
$r_1e^{i\varphi_2}$ back to $r_1e^{i\varphi_1}$ clockwise along the circle	 $|z|=r_1$.

\begin{center}
\begin{minipage} {0.5\textwidth}
\begin{tikzpicture}
\path (0,0) coordinate (P0);
\path (12em,0) coordinate (Q0);
\path (3em,0) coordinate (Q1);
\path (5em,0) coordinate (Q2);

\path (40:10em) coordinate (P1);
\path (40:14em) coordinate (P2);
\path (45:12em) coordinate (P);
\path (50:14em) coordinate (P3);
\path (50:10em) coordinate (P4);

\path (40:12em) coordinate (M1);
\path (50:12em) coordinate (M2);
\path (45:10em) coordinate (M3);
\path (45:14em) coordinate (M4);

\draw (P) circle (0.1em); 
\draw (P1) circle (0.1em); 
\draw (P2) circle (0.1em); 
\draw (P3) circle (0.1em); 
\draw (P4) circle (0.1em); 
\draw (P0) -- (Q0); 
\draw (P0) -- (P2);
\draw (P0) -- (P3);
\draw (P2) arc(40:50:14em); 
\draw (P1) arc(40:50:10em); 
\draw (Q1) arc(0:40:3em);
\draw (Q2) arc(0:50:5em);

\draw (M1) -- ++(-0.1em,-0.35em);
\draw (M1) -- ++(-0.35em,0.05em);
\draw (M2) -- ++(0.01em,+0.35em);
\draw (M2) -- ++(0.35em,0.05em);
\draw (M3) -- ++(-0.35em,-0.05em);
\draw (M3) -- ++(-0.05em,+0.35em);
\draw (M4) -- ++(0.35em,0.05em);
\draw (M4) -- ++(0.05em,-0.35em);

\path (33:11em) node (P10) {$r_1e^{i\varphi_1}$};
\path (35:15em) node (P20) {$r_2e^{i\varphi_1}$};
\path (56:10em) node (P40) {$r_1e^{i\varphi_2}$};
\path (54:14em) node (P30) {$r_2e^{i\varphi_2}$};
\path (45:13em) node (PP)  {$re^{i\varphi}$};
\path (20:2em) node (R1) {$\varphi_1$};
\path (23:4em) node (R2) {$\varphi_2$};
\path (45:15em) node (G)  {$\gamma$};

\end{tikzpicture}\\
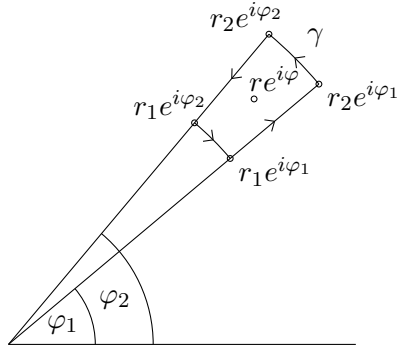
Fig.~1. The path $\gamma$.
\end{minipage}  
\end{center}

In Equation \eqref{10.1} we replace $\mu$ by $\mu/t$, $t>0$,  and integrate the identity over 
the contour $\gamma$. Since $\Tr(R_gT_wA(H+\mu^\fm)^{-K})$ is holomorphic in $\mu$
and the remainder is $o(1)$, the corresponding integrals on the right hand side will tend to zero as $t\to 0^+$. Hence also the left hand side will tend to zero. 
Clearly $\int_\gamma h(\omega) \, d\mu$ is independent of $t$. 
So let us  compute the terms $\int_\gamma e^{i\mu_l|\mu|^2/t^2}g_l(\omega) d\mu$. 

With the obvious parametrization of the four contours defining  $\gamma$ we obtain
\begin{eqnarray}\lefteqn{
\int_\gamma e^{i\mu_l|\mu|^2/t^2}g_l(\omega) d\mu}\nonumber\\
&=&g_l(\varphi_1)e^{i\varphi_1} \int_{r_1}^{r_2} e^{i\mu_l s^2/t^2} ds
+ e^{i\mu_lr_2^2/t^2}\int_{\varphi_1}^{\varphi_2} g_l(s) ir_2e^{is}\, ds \nonumber\\
&&-g_l(\varphi_2) e^{i\varphi_2}\int_{r_1}^{r_2} e^{i\mu_l s^2/t^2} ds\\
&&-e^{i\mu_lr_1^2/t^2}\int_{\varphi_1}^{\varphi_2} g_l(\varphi_1+\varphi_2-s) ir_1e^{i(\varphi_1+\varphi_2-s)}\, ds
\nonumber\\
&=&\left(g_l(\varphi_1)e^{i\varphi_1}- g_l(\varphi_2)e^{i\varphi_2}\right) \int_{r_1}^{r_2} e^{i\mu_l s^2/t^2} ds\label{10.2}\\
&&+ ir_2e^{i\mu_lr_2^2/t^2}\int_{\varphi_1}^{\varphi_2} g_l(s) e^{is}\, ds\label{10.2a} \\
&&-ir_1e^{i\mu_lr_1^2/t^2}\int_{\varphi_1}^{\varphi_2} g_l(\varphi_1+\varphi_2-s) e^{i(\varphi_1+\varphi_2-s)}\, ds
\label{10.3}\end{eqnarray}
As $t\to 0^+$, the term \eqref{10.2} tends to zero.  
We therefore conclude that  the sum over all $l$ of the terms in \eqref{10.2a} and  \eqref{10.3} also
attains a limit. 
However, as the $\mu_l$ are all different and non-zero, the exponentials 
$e^{i\mu_lr_1^2}$ and $e^{i\mu_lr_2^2}$ will be all different for almost all  
choices of $r_1$ and $r_2$. Hence the existence of a limit as $t\to 0^+$ implies 
that all coefficients vanish, i.e. 
$$ \int_{\varphi_1}^{\varphi_2} g_l(s) e^{is}\, ds=0= \int_{\varphi_1}^{\varphi_2} g_l(\varphi_1+\varphi_2-s) e^{i(\varphi_1+\varphi_2-s)}\, ds$$
for all $l$ and all choices of $\varphi_1$ and $\varphi_2$.
This in turn shows that all $g_l$ are zero near $\omega = \varphi$. 
 From this we deduce that 
 $$\int_\gamma   h(\omega) d\mu=0$$
for all these contours $\gamma$.
Since the point $z$ was arbitrary, we conclude that all the $g_l$ are zero and $h$ is holomorphic and - as it only depends on $\omega$ -  even constant. 
Hence the coefficient of the leading order term $\mu^d\ln^k\mu$ is a constant. We can then apply the same conclusion iteratively to the lower order terms. 
This completes the argument.  
\end{proof}

\begin{remark}
From another point of view, the terms with $e^{i|\mu|^2\mu_l}$  arise at the sphere $|\xi| =  |\mu|$ in the integrals with respect to $\xi$  over the areas $1\le |\xi|\le |\mu|$ and $|\mu|\le |\xi|$ and therefore should cancel.
\end{remark}

For the proof of Theorem {\rm \ref{thm_resolvent}} 
it remains to note that by   \eqref{phase3} and \eqref{eq-diag2}), $\varkappa_1=2\dim  \mathbb{C}^n_g$ is the real dimension of the fixed point set of $g\in \U(n)$.

\section{Complex Powers and Heat Trace Expansions}

As a consequence of Theorem {\rm \ref{thm_resolvent}}, we obtain two more results.
We use the notation of Theorem {\rm \ref{thm_resolvent}}, in particular \eqref{sector} for $S_\delta$. 

\begin{theorem}\label{thm_zeta}
Assume, moreover,  that  $H-\gl$ is invertible for all $\gl\in S_{\gd}\cup U_r(0)$ for some $r>0$. Then 
\begin{enumerate} \renewcommand{\labelenumi}{(\alph{enumi})}
\item The function 
$z\mapsto \zeta_{R_gT_wA}(z):=\Tr(R_gT_wAH^{-z})$ is  holomorphic in the half-plane   $\{z\in \C: \fm \re z > 2n+\fa\}$.

\item $\zeta_{R_gT_wA}$ has a meromorphic continuation to $\C$ with at most simple poles 
in the points $(2m+\fa-j)/\fm$ and $z=-j$, $j=0,1,\ldots$, where $m$ is the complex dimension of the
fixed point set of $g$. 
The function $\Gamma \zeta_{R_gT_wA}$ has the pole structure \eqref{trace_exp}.

\item The coefficients $\tilde c_j, \tilde c_j'$ and $\tilde c_j''$ in \eqref{trace_exp}  are related to the coefficients $c_j, c_j'$ and $c_j''$ in \eqref{resolvent_exp} by universal constants. In particular, we have 
$$\Res\zeta_{R_gT_wA} = c_0' = \tilde c_0'.$$ 
 
\item If the fixed point set of the affine mapping $\mathbb{C}^n\to\mathbb{C}^n, v\mapsto gv+w$ is empty, then $\zeta_{R_gT_wA}$ has no poles.

\item $\zeta_{R_gT_wA}$ has rapid decay along vertical lines $z=c+it$, $t\in \mathbb R:$
\begin{equation}\label{eq-decay3}
 |\zeta_{R_gT_wA}(z)\Gamma(z)|\le C_N (1+|z|)^{-N},\quad \text{for all }N\ge 0, 
 \quad |\im z|\ge 1,
\end{equation}
uniformly for $c$ in compact intervals.
\end{enumerate} 
\end{theorem}

With Theorem \ref{thm_resolvent}  established, the  crucial observation for the proof of Theorem \ref{thm_zeta} is the proposition, below. 
It is adapted from \cite[Proposition 2.9 and Corollary 2.10]{GS2}  for the convenience of the reader.

\def\Zeta{Z} 
\begin{proposition}\label{GrubbS}
{\rm(a)} Let $0<\delta_0\le\pi$, $r_0>0$, and  let $f: U_{r_0}(0) \cup S_{\delta_0}\to \mathbb C$ be a meromorphic function with a Laurent expansion near  $\lambda =0$:
$$f(\gl) = \sum_{j=-k}^\infty b_j(-\gl)^{j}. $$
Assume, moreover,  that  $f(\lambda) = O(|\lambda|^{-\alpha})$ for some $\alpha\in{]0,1]}$ as $|\lambda|\to \infty$ in each sector $S_{\delta}$ for $\delta<\delta_0$. Define 
\begin{eqnarray}\label{GrubbS.1}%
\Zeta(s) = \frac{i}{2\pi} \int_{\scrC_r} \lambda^{-s} f(\lambda) \, d\lambda, \ \re s>1-\alpha,
\end{eqnarray}
for the contour $\scrC_{r}$, $r<r_0$, in $\mathbb C$ going from infinity to $-r$  on the ray 
$\{se^{i\pi} : s\ge r\}$, clockwise about the origin on the circle of radius $r$ and back to infinity along the ray $\{se^{-i\pi}: s\ge r\}$. ,  described above.  
Then the function $\frac{\pi\Zeta(s)}{\sin \pi s}$ is meromorphic for $\re s>1-\alpha$
and has at most simple poles in $s=j+1$ and residues $(-1)^{j+1} \Zeta(j+1)=:-b_j$, $j=0,1,\ldots$. 

Moreover, the following are equivalent:
\begin{enumerate}
\renewcommand{\labelenumi}{{\rm (\roman{enumi})}}
\item For every $\delta<\delta_0$, $f$ has an asymptotic expansion as $\lambda$ goes to infinity 
\begin{eqnarray}\label{GrubbS.2}%
f(-\lambda)  \sim \sum_{j=0}^\infty \sum_{l=0}^{m_j} a_{j,l} \lambda^{-\alpha_j}\ln^l\lambda ,
\end{eqnarray}
with $0<\alpha_{j}\nearrow +\infty$ as $j\to \infty$ and $m_j\in \mathbb N_0$,
uniformly for $-\lambda\in S_\delta$. 

\item $\frac{\pi\Zeta(s)}{\sin \pi s}$ extends meromorphically to $\mathbb C$ with the singularity structure 
\begin{eqnarray}
\frac{\pi\Zeta(s)}{\sin \pi s} \sim
-\sum_{j=-k}^\infty \frac{b_j}{s-j-1} + 
\sum_{j=0}^\infty \sum_{l=0}^{m_j} \frac{\alpha_{jl}l!} {(s+\alpha_j-1)^{l+1}},
\end{eqnarray}  
and for each real $C_1,C_2$ and each $\delta<\delta_0$ 
\begin{eqnarray}\label{GrubbS.3}%
\left|\frac{\Zeta(s)}{\sin \pi s} \right| \le Ce^{-\delta|\im s| } , \quad |\im s|\ge 1, C_1\le \re s\le C_2, 
\end{eqnarray}
where $C$ depends on $C_1$, $C_2$, and $\delta$. 
\end{enumerate}     

{\rm (b)}  When $f$ and $\Zeta$ are as in {\rm (a)}, then $\Zeta\Gamma$ is meromorphic on $\C$ with the singularity structure
$$\Gamma(z)\Zeta (z)\sim \sum_{j=-k}^{-1} \frac{-\tilde b_j}{z-j-1} 
+ \sum_{j=0}^\infty \sum_{l=0}^{m_j} \frac{\tilde a_{jl}l!}{(z+\ga_j-1)^{l+1}} , \quad \tilde b_j = \frac{b_j}{\gG(-j)}, \tilde a_{jl} = \frac{a_{jl}}{\gG(-\ga_j)}.  
$$ 
When $\gd_0>\pi/2$ one has moreover, for any $\gd'<\gd_0-\pi/2$ and any real $C_1$ and $C_2$:
$$
|\Gamma(s)\Zeta(s)| \le C(C_1,C_2,\gd') e^{-\gd' |\im z|}  ,\quad  |\im z |\ge 1, C_1\le\re z \le C_2. 
$$
\end{proposition}

{\em Proof} of Theorem \ref{thm_zeta}.
We apply Proposition \ref{GrubbS}(a)  to the function 
$$f(\lambda) = \Tr (R_gT_wA(H-\lambda)^{-K}).  
$$
The invertibility of $H-\gl$ implies that $f$ is holomorphic in $U_r(0)\cup S_\gd$. 
In particular, the Laurent coefficients  $b_{-k}, \ldots, b_{-1}$ all vanish. 
Theorem \ref{thm_resolvent} guarantees the decay at infinity and the existence of the asymptotic expansion for $f(-\gl)$ as $\gl\to \infty$ in  $S_\gd$.  

So the assumptions in Proposition \ref{GrubbS} are fulfilled and from (b) we obtain the pole structure of $\Gamma (z) \Zeta(z)$ as well as the decay along vertical lines. 

Next, integration by parts shows that 
\begin{eqnarray*}
\zeta_{R_gT_wA}(z+K-1) =  (-1)^{K-1}  \frac{(K-1)!}{z\cdots (z+K-2)} \, Z(z), \quad K\ge2, 
\end{eqnarray*} 
and so 
\begin{eqnarray*}
\lefteqn{
\zeta_{R_gT_wA} (s) \gG(s) = \frac{(-1)^{K-1}(K-1)!}{(s-K+1)\cdots (s-1)}Z(s-K+1)\gG(s)}\\ 
&=& 
(K-1)!\ Z(s-K+1)\gG(s-K+1).
\end{eqnarray*}
Hence the poles of $\zeta_{R_gT_wA} (s)\gG(s)$ are just those of $\Zeta(s)\gG(s)$ shifted by $K-1$; the residues differ by universal factors. We obtain the statements  in Theorem 
\ref{thm_zeta}. \hfill $\Box$
\medskip

\begin{theorem}\label{thm_exp}Assume additionally that $H-\gl$ is invertible for $\gl\in S_{\gd_0}\cup U_{r_0}(0)$ for some $\gd_0\in {]\pi/2, \pi]} $, $r_0>0$. Then
\begin{enumerate} \renewcommand{\labelenumi}{(\alph{enumi})}
\item The function $t\mapsto \Tr(R_gT_wA\exp(-tH)) $ is defined for all $t>0$. 
\item As $t\to 0^+$ it has the asymptotic expansion \eqref{exp_exp}
with the same coefficients as for the zeta function in Theorem {\rm \ref{thm_zeta}}. 
\end{enumerate}
\end{theorem} 

The proof relies on the following result of  
Grubb and Seeley \cite[Proposition 5.1]{GS2} that we state here for convenience: 

\begin{proposition} \label{GrubbSe2}
{\rm (a)}  Let $e(t)$ be  function holomorphic in a sector 
$V_{\gt_0}=\{z=re^{i\phi}: r>0,|\phi|<\gt_0\}$  for some $\gt_0\in\ ]0, \pi [$,
such that $e(t)$ decreases exponentially for $|t| \to  \infty$ and is $O(|t|^a)$ for $t \to  0$ in $V_\gt$,  
any $\gt<\gt_0$, for some $a\in \R$. Let $f$ be the Mellin transform of $e$,
\begin{eqnarray}
\label{(5.6)}
 f (s) = (\cM e)(s) := \int_0^\infty t^{s-1} e(t) \,dt,\quad  \re s > -a.
\end{eqnarray}
Then $f(s)$ is holomorphic for $\re s > -a$ and $f(c + i\xi)$ is $ O(e^{-\gt|\xi|})$ for $|\xi| \to  \infty$, 
when $c > -a$, uniformly for $c$ in compact intervals of $ ]-a, \infty[$; 
and $e(t)$ is recovered from $f(s)$ by the formula
\begin{eqnarray}
\label{(5.7)}
e(t) = \frac1{2\pi i} \int_{\re s=c}t^{-			s}f(s)\,ds.
\end{eqnarray}

{\rm (b)}  Moreover, the following properties {\rm(i)} and  {\rm(ii)} are equivalent:

{\rm (i)} $e(t)$ has an asymptotic expansion for $t \to  0$,
\begin{eqnarray}
\label{(5.8)}
 e(t) \sim\sum_{j=0}^\infty \sum_{l=0}^{m_j} 
 a_{j, l} t^{\beta_j} (\ln t)^l, \quad \beta_j \to \infty, 
 m_j \in  \{0,1,2,\ldots\},
\end{eqnarray}
uniformly for $t \in V_\gt$, for each $\gt < \gt_0$.

{\rm (ii)} $ f(s)$ is meromorphic on $\C$ with the singularity structure
\begin{eqnarray}
\label{(5.9)}
 f(s) \sim\sum_{j=0}^\infty\sum_{l=0}^{m_j} \frac{(-1)^ll!a_{j,l}}{(s+\beta_j)^{l+1}}
\end{eqnarray}
and for each real $C_1, C_2$ and each $\gt < \gt_0$,
\begin{eqnarray}
\label{(5.10)}
 |f(s)| \le C(C_1,C_2,\gt) e^{-\gt |\im s|}, \quad |\im s| \ge 1, C_1 \le \re s \le C_2.
\end{eqnarray}

{\rm (c)} 
Let $r(\lambda)$  be holomorphic in $S_{\gd_0} = \{|\pi-\arg \lambda| < \gd_0\}$ for some $\gd_0 \in {]\frac\pi2,\pi]}$ with values in a Banach space and meromorphic at $\lambda = 0$ 
(holomorphic for $0 < |\lambda| < \rho$). 
Assume that as $\lambda\to \infty$ in $S_\gd$ {\rm(}for $\gd<\gd_0)$, $\partial^j_\lambda  r(\lambda)$ is 
$O(|\lambda|^{-1-\epsilon})$ for some $\epsilon>0$ {\rm(}so that $r(\lambda)$ is
$O(|\lambda|^{j-1}))$. 
Let $\gt_0$  and $\gt$ be such that ${]\gt-\gt_0,\gt+\gt_0[}\subset {]\pi-\gd_0, \frac\pi2[}$, let $\scrC = \scrC_{\gt,r_0}$ be the contour given by the clockwise oriented boundary of the set  
$$\{z\in \C: |\arg z| \ge\gd\} \cup U_{r_0}(0)$$
with $r_0 \in{]0,\rho[}$, and let 
\begin{eqnarray}
\label{(5.11)}
 e(t) = \frac i{2\pi} \int_\scrC e^{-t\lambda} r(\lambda) \, d\lambda, \quad f(s) = \Gamma(s) \frac i{2\pi}  \int_\scrC\lambda^{-s}r(\lambda)\, d\lambda, 
\end{eqnarray}
for $t \in  V_{\gt_0}$ resp. $\re s > j - \epsilon$. Then $e(t)$ is exponentially decreasing for $t \to  \infty$ in sectors $V_\gd$ with $\gd < \gt_0$, and is $O(|t|^{-j})$ for $t \to  0$, 
and $f(s)$ and $e(t)$ correspond to one another by \eqref{(5.6)}, \eqref{(5.7)}.

\end{proposition}

{\em Proof} of Theorem \ref{thm_exp}. We consider the function 
$$r(\gl) = R_gT_wA(H-\gl)^{-1}. $$ 
By our assumption, it is holomorphic for  $\gl\in S_{\gd_0}\cup U_{r_0}(0)$ for some $\gd_0 \in {]\frac\pi2,\pi]}$ and small $r_0>0$, taking values in $\scrB(L^2(\R^n), \cH^{\fm-\fa}(\R^n))$, uniformly in $\gl$.  
Moreover,  $\partial^j_\gl r(\gl) = (j-1)!\ R_gT_wA(H-\gl)^{-j-1}$. 
Hence, for $j$ large enough,  $\partial_\gl^j r(\gl)$ will be a family of bounded operators on $L^2(\R^n)$, and $O(\gl^{-1-\gve})$  for suitable $\gve>0$ as $\gl\to \infty$. Increasing $j$ even further, we obtain a bounded family of trace class operators that is $O(\gl^{-1-\gve})$.  

Proposition  \ref{GrubbSe2}(c) therefore implies that we can define the functions $e(t)$ and $f(s)$ in 
\eqref{(5.11)}.  Moreover, it follows that $e(t)$ is exponentially decreasing for $t\to \infty$ in a  sector 
$V_{\gt_0}$ about the positive real axis (with the notation from \ref{GrubbSe2}(c)). 
Since  we have for arbitrary $k\in \N$ 
$$ e(t) = \frac i{2\pi} \int_{\scrC}e^{-t\gl}r(\gl) \, d\gl =\frac {it^{-k}}{2\pi} \int_{\scrC}e^{-t\gl} 
\partial^k_\gl r(\gl)\, d\gl ,
$$
we can differentiate $e$ 
with respect to $t$ as a complex variable, and see that 
$t\mapsto e(t)$ is even holomorphic in $V_{\gt_0}$. Hence the same is true for 
$$t\mapsto \Tr(e(t))= \Tr(R_gT_wA e^{-tH}).
$$
So the assumptions of Proposition \ref{GrubbSe2}(a) are fulfilled and we deduce from 
part (b)  and Theorem \ref{thm_zeta} the statements of Theorem \ref{thm_exp}. 
\hfill $\Box$

\section{The Noncommutative Residue and Equivariant Traces} 
Let $G$ be a  subgroup of $\C^n\rtimes \U(n)$ and 
$$D=\sum_{(w,g) \in G} R_gT_wA \in \scrA,$$ 
where the sum is finite.

\begin{definition}\label{def_ncr}
For $w_0\in \C^n$, $g_0\in \U(n)$ such that $(w_0,g_0)\in G$ we define the noncommutative residue 
$\res_{\skp{(w_0,g_0)}} $
localized at the conjugacy class $\skp{(w_0,g_0)}$ 
in $G$ for an operator  $D$ as above  by 
\begin{eqnarray*}\label{def_res}
\lefteqn{\res_{\skp{(w_0,g_0)}} D = \ord H
\sum_{(w,g)\in \skp{(w_0,g_0)}}\Res \zeta_{A,g,w}(z)}\\
&=& \ord H \sum_{(w,g)\in \skp{(w_0,g_0)}} c_0' ( R_gT_wA)
=\ord H \sum_{(w,g)\in \skp{(w_0,g_0)}}\tilde c_0' ( R_gT_wA)
\end{eqnarray*} 
with the coefficients $c_0'$ and $\tilde c_0'$ introduced in Theorem
\ref{thm_resolvent} and Theorem \ref{thm_zeta}, respectively, 
and the order $\ord H$ of the auxiliary operator $H$ used in the construction of the trace expansion. 
Note that the noncommutative residue is independent of $H$ by Theorem \ref{thm_resolvent}(e), see also Lemma \ref{4.16}. 
\end{definition} 

\begin{theorem}\label{3.8} 
For every choice of a conjugacy class $\skp{(w_0,g_0)}$, the  functional $\res_{\skp{(w_0,g_0)}}$ is  a trace.
\end{theorem}  

\begin{proof}This was shown in \cite[Theorem 3.8]{SaSch4}.\end{proof} 

In principle, the noncommutative residue can be computed explicitly:
By linearity, it is sufficient to consider a single operator $R_gT_wA$. 
Moreover, as we saw in Lemma \ref{g},  we may assume that $g$ is diagonalizable and of the form \eqref{eq-diag2}. Then we obtain:

\begin{theorem}\label{residue}
Let  $D=R_gT_wA$,  $g$ as in \eqref{eq-diag2}, 
$w=a-ik\in \C^n$ and  $A\in \Psi^{\fa}$. For $\fa=-2m_5$, 
\begin{eqnarray}
\nonumber
\res_{\skp{(w,g)}} D 
= C_{\rm ncr}  
 \int_{\mathbb S^{2m_5-1}}\tr_E a_{-2m_5} (\theta)\,dS.
 \label{eqn_res}
\end{eqnarray}
Here, 
$\mathbb S^{2m_5-1}$ denotes the unit sphere in the $+1$-eigenspace of  
$g$, and 
$a_{-2m_5}$ is the  component of homogeneity  $-2m_5$ in the symbol  expansion of $A$.
The constant is explicitly given by 
\begin{eqnarray*}
C_{\rm ncr} &=&
2^{-n}\pi^{-m_5} 
\prod_{j=1}^{m_1+m_2+m_3}
 e^{\frac i4\ctg\frac{\varphi_j}{2}(k_j^2+a_j^2)}
\Big(1-i\ctg\frac{\varphi_j}2\Big).
\end{eqnarray*}
For  $A\in \Psi^\fa$ with $\fa<-2m_5$ the residue vanishes, while 
additional terms containing derivatives of components $a_{\fa-j}$, $\fa-j>-2m_5$ enter for $\fa>-2m_5$. Details are given, below. 
\end{theorem}

\begin{proof} Replacing $-\gl$ by $\mu^\fm$, we are interested in the $\mu^{-K\fm}\ln\mu$-term in \eqref{log-term}.  
In the proof of Lemma \ref {4.16} we saw that it will only appear  if, in the notation of \eqref{4.12.3}, we have $\ell=k=0$.
 
By Proposition \ref{4.2}(b), the only terms of the form $d(y,p)(h_\fm(y ,p)+\mu^\fm)^{-K}$ 
in the symbol expansion of $A(H+\mu^\fm)^{-K}$ with $d$ homogenous for $|(x, p)|\ge1$ are those, where  $d$ is one of the homogeneous components of the symbol of $A$. 
In addition, \ref{4.2}(d) showed that the only factors $d$ associated to $(h_\fm(\tilde Bw)+\mu^\fm)^{-K}$ in the expansion of $\tilde q(\tilde Bw-\tilde b_0)$ are the functions 
\begin{eqnarray}\label{bq}
\frac{(-\tilde b_0)^\ga}{\ga!}  \partial_w^\ga \tr_E a_{\fa-j}(\tilde Bw)
\end{eqnarray}
for some choice of $j$ and $\ga$. Further derivatives on the $a_{\fa-j}$  might come from the stationary phase expansion in \eqref{4.14.1a}.
We next distinguish the cases where $\fa=-2m_5$, $\fa<-2m_5$ or $\fa>-2m_5$. 
 
(i) Let  $\fa=-2m_5$. Since derivatives decrease the order, a nontrivial contribution to the $\mu^{-K\fm}\ln \mu$-term is only possible, if $j=|\ga|=0$ in \eqref{bq} and $J=0$ in \eqref{log-term}.
But then, \eqref{log-term} and \eqref{4.14.1a} show that the coefficient is 
\begin{eqnarray}\label{cont0}
c_{1,0,0}=(2\pi)^{(n'-2m_5)/2}\frac{e^{i\frac\pi4\sgn Q}}{|\det Q|^{1/2}}
\int_{\mathbb S^{2m_5-1}} \tr_E a_{-2m_5}(\tilde B(\gt)) dS
\end{eqnarray}
with $Q={\diag} (2\tilde \gl_1, \ldots, 2\tilde \gl_{n'-2m_5})$ and the $\tilde \gl_j$ introduced after \eqref{eq:simple} so that 
\begin{align*}
&|\det Q|=\prod_{j=1}^{n'-2m_5}|2\tilde \gl_j|=2^{m_2+m_3+2m_4}\prod_{j=1}^{m_1}|2\gl^+_j 2\gl^-_j|\\
&=2^{ m_2+m_3+2m_4}\prod_{j=1}^{m_1} 
\left|\frac{\sin\varphi_j-(1-\cos\varphi_j)}{\cos\varphi_j}
\frac{\sin\varphi_j+(1-\cos\varphi_j)}{\cos\varphi_j}\right|\\
&=2^{ m_2+m_3+2m_4}\prod_{j=1}^{m_1} 
\frac{2|\cos\varphi_j|(1-\cos\varphi_j)}{\cos^2\varphi_j}
=2^{2m_1+m_2+m_3+2m_4}\prod_{j=1}^{m_1}\frac{\sin^2\frac{\varphi_j}2}{|\cos\varphi_j|}
\end{align*}
and 
\begin{equation*}
\sgn Q=m_3-m_2-\sum_{j,\pm}\sgn \lambda_j^\pm =m_3-m_2- \sum_j\sigma(\varphi_j),
\end{equation*}
where
$$
\sigma(\varphi_j)=\sgn \lambda_j^++\sgn\lambda_j^-=\left\{
\begin{array}{ll}
2\sgn \varphi_j, & \text{if }|\varphi_j|<\pi/2\\
0, & \text{otherwise}.
\end{array}
\right.
$$
The second identity holds, since $\lambda^+_j$ is negative only for $\varphi_j\in(-\pi/2,0)$, while $\lambda_j^-$ is positive only for $\varphi_j\in(0,\pi/2)$.

We can now express the noncommutative residue as:
$$
\res_{\skp{(w,g)}} D=C_{\rm res} (2\pi)^{(n'-2m_5)/2}\frac{e^{i\frac\pi4\sgn Q}}{|\det Q|^{1/2}}\int_{\mathbb S^{2m_5-1}} \tr_E a_{-2m_5}(\gt)\, dS.
$$
Using  the equalities:
\begin{align*}
-\sigma(\varphi_j)-\sgn\varphi_j+\sgn\ctg\varphi_j&=-2\sgn\varphi_j;\\
e^{i\frac\pi4(-\sigma(\varphi_j)-\sgn\varphi_j+\sgn\ctg\varphi_j)}&=e^{-i\frac\pi2\sgn\varphi_j}=-i\sgn\varphi_j\\
 1-i\ctg(\pm\pi/4)&= \sqrt 2\;e^{\mp i\frac \pi 4}
\end{align*}
the prefactor can be simplified as follows:
\begin{align*}
&(2\pi)^{-n+\frac{m_2+m_3}2} 
\prod_{j=1}^{m_1}  
\frac{e^{i\frac{\varphi_j} 2}}{|\cos\varphi_j|^{1/2}}e^{i\frac\pi 4(-\sgn\varphi_j+\sgn\ctg\varphi_j)}
   \prod_{j=1}^{m_1+m_2+m_3}
 e^{\frac i4\ctg(\varphi_j/2)(k_j^2+a_j^2)}\\
& \times
 (2\pi)^{n-\frac{m_2+m_3}2-m_5}e^{i\frac\pi4(m_3-m_2)}
 2^{-m_1-\frac{m_2+m_3}2-m_4}\prod_{j=1}^{m_1}\frac{|\cos\varphi_j|^{1/2}}{|\sin\frac{\varphi_j}2|}e^{-i\frac\pi4\sigma(\varphi_j)}\\
& =2^{-m_1-\frac{m_2+m_3}2-m_4-m_5}\pi^{-m_5}e^{i\frac\pi4(m_3-m_2)} \prod_{j=1}^{m_1+m_2+m_3}
 e^{\frac i4\ctg\frac{\varphi_j}{2}(k_j^2+a_j^2)}\prod_{j=1}^{m_1}  
\Big(1-i\ctg\frac{\varphi_j}2\Big)\\
&
=2^{-n}\pi^{-m_5} 
\prod_{j=1}^{m_1+m_2+m_3}
 e^{\frac i4\ctg\frac{\varphi_j}{2}(k_j^2+a_j^2)}
\Big(1-i\ctg\frac{\varphi_j}2\Big).
\end{align*}

We know that $\tilde B$ is the identity on $(x_{n-m_5+1},p_{n-m_5+1},\ldots 
x_n,p_n)$ and that, on $\mathbb S^{2m_5-1}$, the other variables are zero. Hence the integral in~\eqref{cont0} reduces to    
\begin{eqnarray}\label{cont1}
\int_{\mathbb S^{2m_5-1}} \tr_E a_{-2m_5}(\gt)\, dS
\end{eqnarray}
as asserted. 

(ii) If $\ord A<-2m_5$, then there will be no term of the form $d(y,p)(h_\fm(y,p)+\mu^\fm)^{-K}$ 
where $d$ is both of homogeneity $-2m_5$ and a derivative of a symbol component of $a$.
According to the above considerations,  there will be no contribution to the noncommutative residue.  

(iii) Let $\ord A>-2m_5$. Following the above approach, we would have to treat also terms arising from the stationary phase expansion \eqref{4.14.1a} for $J\not=0$, which are not explicitly computable. 
Fortunately, this can be avoided: We go one step back and reconsider Equation \eqref{4.12.4}, 
using a special parametrization of a tubular neighborhood of 
$$\mathbb S^{2m_5-1} =\{ \gt \in \mathbb S^{n'-1} : 
\gt' = (\gt_1, \ldots , \gt_{n'-2m_5})=0\} \subset \mathbb S^{n'-1}:$$ 
Namely, we fix a parametrization 
$\Phi: D\subset \R^{2m_5-1}\to \mathbb S^{2m_5-1}$ (e.g. the usual spherical coordinates) 
and let, for small $\gve>0$ and $W_\gve = [-\gve,\gve]^{n'-2m_5}$,
$$F: W_\gve\times D\to \R^{n'}, \quad F(\gt',\go) = (\gt', \Phi(\go)\sqrt{1-|\gt'|^2}).$$
The induced measure is $(1-|\gt'|^2)^{m_5-1} d\gt'dS$ with the surface measure $dS$ of
$\mathbb S^{2m_5-1}$.

We know that there will be no contribution to the $\mu^{-K\fm}\ln\mu$-term  from \eqref{4.12.4} 
unless $k=0$ there.
 Up to terms that are $O(|\mu|^{-\infty})$, the integral, with $k=0$, $l=1$ and $d$ from  \eqref{bq} of homogeneity degree $\fd = \fa-j-|\ga|$,  is
\begin{eqnarray}
\label{oscint}
\int_1^{|\mu|} \int_D\int_{W_\gve} e^{ir^2\skp{Q\gt',\gt'}/2} d(\gt',\Phi(\go)\sqrt{1-|\gt'|^2}) \, d\gt'dSr^{\fd+n'-1}\, dr,
\end{eqnarray}
where $Q={\diag} (2\tilde \gl_1, \ldots, 2\tilde \gl_{n'-2m_5})$ with the $\tilde \gl_j$ introduced after \eqref{eq:simple}.

The advantage is that we may now  apply to the inner integral the explicit version of the stationary phase expansion for quadratic exponentials,   \cite[Theorem 3.13]{Z}, with $h=r^{-2}$: 
\begin{eqnarray*}
\lefteqn{\nonumber
\int e^{\frac ih \skp{Q\gt',\gt'}/2} d(\gt', \Phi(\go)\sqrt{1-|\gt'|^2}) \, d\gt'}\\
&=& \left( {2\pi h}\right)^{\frac{n'-2m_5}2} \frac{e^{i\frac\pi4\sgn Q}}{|\det Q|^{1/2}}   
\left(
\sum_{J=0}^{N-1} \frac{h^J}{J!} \left( \frac{\skp{Q^{-1}D,D}}{2i}\right)^J d(0,\Phi(\go)) +O(h^N) \right).
\end{eqnarray*}
Taking $h=r^{-2}$,  this furnishes an expansion into decreasing powers of $r$, which we then insert into   \eqref{oscint}. 
When evaluating the integral over $[1,|\mu|]$, we will only get a nontrivial contribution to the $\mu^{-K}\ln \mu$-term, if the total $r$-power is $-1$. 
In the case at hand, the possible powers are $(\fd+n'-1) +(-n'+2m_5)-2J = \fa-j-|\ga|+2m_5-2J-1$. Hence only those (finitely many) terms, where
$$\fa-j-|\ga| -2J = -2m_5$$
can contribute. For these we obtain an explicit formula. 
\end{proof}

\section{Appendix}

\subsection*{Weakly parametric pseudodifferential operators}

Following \cite[Section 1]{GS1} we let  $S$ be an open sector in $\mathbb C$ 
with vertex at the origin and  $p=p(x,\xi,\mu)$ a smooth function on 
$\mathbb R^\nu\times \mathbb R^n\times
\overline{S}$, which is additionally holomorphic for $\mu\in S$ and $|\xi,\mu|\ge \gve$, for some $\gve>0$. 

We say that  $p\in S^{\fm,d}(\mathbb R^\nu,  \mathbb R^n;S)$  
provided that 
$$\partial^j_t( t^dp(\cdot,\cdot, 1/t)) \in S^{m+j} (\mathbb R^\nu\times  \mathbb R^n) $$
for $ 1/t \in S$, with uniform estimates in 
$S^{\fm+j} (\mathbb R^\nu\times  \mathbb R^n) $ for $|t|\le 1$ and $1/t$ in closed 
subsectors of $S$.  

We shall need a few basic facts:

\begin{lemma}\label{basic}
{\rm (a)} If $p=p(x,\xi)\in S^\fm(\mathbb R^\nu\times \mathbb R^n)$ is independent of $\mu$, 
then $p\in S^{\fm,0}( \mathbb R^\nu, \mathbb R^n; \mathbb C)$, see \cite[Example 1.2]{GS1}.
\\
{\rm (b)} If $p_1\in S^{\fm_1,d_1}$ and $p_2\in S^{\fm_2,d_2}$, then 
$p_1p_2\in S^{\fm_1+\fm_2,d_1+d_2}$, see \cite[Lemma 1.6]{GS1}. 
\\
{\rm (c)} If $a_\fm=a(x,\xi)$ is homogeneous of integer degree $\fm>0$ in $\xi$ for $|\xi|\ge1$ and smooth in $(x,\xi)$ on $\mathbb R^\nu\times \mathbb R^n$, and if $a(x,\xi)+\mu^\fm$ is invertible for $(x,\xi,\mu) \in \R^\nu\times\R^n\times \overline S$, then the symbol 
$p$ defined by 
$$p(x,\xi,\mu) = (a(x,\xi) +\mu^\fm)^{-1}$$
is an element of $ S^{0,-\fm}(\R^\nu, \R^n; S)\cap S^{-\fm,0}(\R^\nu, \R^n; S) $ by
\cite[Theorem~1.17]{GS1}.  
\\
{\rm (d)} If $a$ is as in {\rm (c)} and $a(x,\xi)+\mu^\fm$ is invertible for $(x,\xi,\mu) \in \R^\nu\times\R^n\times \overline S$, $|\xi|\ge1$, then  $a$ can be modified on $|\xi|<1$ so that $a_\fm(x,\xi)+\mu^\fm$ is invertible for all $(x,\xi)$, $\mu\in S$, see e.g. Seeley \cite{Seeley67}. 
In the sequel we shall assume that this has been done when considering such symbols. 
\end{lemma}

\begin{theorem}\label{GrubbS2} {\rm (\cite[Theorem 1.12.]{GS1})}
For $p\in S^{\fm,d}$ set 
\begin{eqnarray}\label{exp_infty}%
p_{(d,k)} (x,\xi) = \frac1{k!} \partial^k_t(t^dp(x,\xi,1/t))_{|t=0}.
\end{eqnarray}
Then $p_{(d,k)}\in S^{\fm+k} $, and for any $N$, 
\begin{eqnarray}
p(x,\xi,\mu) -\sum_{k=0}^{N-1} \mu^{d-k} p_{(d,k)} (x,\xi) \in S^{\fm+N, d-N}.
\end{eqnarray}   
\end{theorem}

\begin{theorem}\label{parametrix}
Let $a\in S^\fm_{cl}$, $\fm>0$, with principal symbol $a_\fm$ as in Lemma \ref{basic}(c).  Then the standard parametrix construction furnishes a symbol $p=p(x,\xi,\mu)$ such that 
$$p\circ (a+\mu^\fm)  \sim 1 \sim (a+\mu^\fm)\circ p\text{\rm { mod }} 
S^{-\infty,-\fm}(\R^n, \R^n;S).$$
The symbol $p$ is an element of $S^{-\fm,0}\cap S^{0,-\fm}$ with an asymptotic 
expansion 
$$p\sim\sum_{j=0}^\infty p_{-\fm-j},$$
where
\begin{eqnarray*}
p_{-\fm}(x,\xi,\mu) &=& (a_\fm(x,\xi)+\mu^\fm)^{-1}  \in (S^{-\fm,0}\cap S^{0,-\fm})(\R^n, \R^n;S)  \text{ and} \\
p_{-\fm-j} &\in& (S^{-\fm-j,0}\cap S^{\fm-j,-2\fm})(\R^n, \R^n;S).
\end{eqnarray*}    
\end{theorem}

The construction is standard, letting 
\begin{eqnarray*}%
p_{-\fm} (x,\xi,\mu) &=&(a_\fm(x,\xi)+\mu^\fm)^{-1}\quad\text{and} \\
p_{-\fm-j} (x,\xi,\mu)
&=& \sum_{k, \alpha, \beta}c_{k,\alpha,\beta} p_{-\fm}(\partial_\xi^{\ga_1}\partial_x^{\gb_1} a_{\fm-k_1} ) 
\ldots 
p_{-\fm}(\partial_\xi^{\ga_r}\partial_x^{\gb_r} a_{\fm-k_r} )p_{-\fm},
\end{eqnarray*}
where, for $j>0$, we have $r\ge 1$ and 
\begin{eqnarray*}%
\sum_{l=1}^r|\ga_l| =\sum_{l=1}^r |\gb_l| 
\text{ and } \sum_{l=1}^r|\ga_l|+\sum_{l=1}^r|k_l|=j .
\end{eqnarray*}
Details can be found in \cite[p.~501]{GS1}. Note: All terms in the asymptotic expansion of the symbol of the parametrix are functions of $\mu^\fm$ and hence the coefficients $p_{(-\fm,k)}$ in Theorem 
\ref{GrubbS2} are zero unless $k$ is a multiple of $\fm$, so that the expansion is in powers $\mu^{d-k\fm}$. 

We need a few more simple observations: 

\begin{lemma} \label{4.5}
{\rm (a)} Let $p(\xi,\eta,\zeta; \mu)\in S^{\fm,d}(\R^k\times \R^l\times \R^l;S)$ be independent of $x$.
Define $\tilde p(\xi,\eta; \mu) = p(\xi,\eta,\eta;\mu)$. Then   
$\tilde p \in S^{\fm,d}(\R^k\times \R^l;S)$. 

{\rm (b)} Let $p(\xi; \mu)\in S^{\fm,d}(\R^k;S)$. Define $\tilde p(\xi;\mu) = p(B\xi+b;\mu)$ for $B\in \cL(\R^k)^{-1} $ and $b\in \R^k$. Then $\tilde p \in S^{\fm.d}(\R^k;S)$. 

{\rm (c)}  Let $p(\xi; \mu)\in S^{\fm,d}(\R^k;S)$ be homogeneous in $(\xi,\mu)$ 
for $|\xi|\ge1$ and $b\in \R^k$. 
Define $\tilde p(\xi,\mu)  = p(\xi+b;\mu)$ and consider a finite Taylor expansion 
\begin{eqnarray*}%
p(\xi+b;\mu) &=& \sum_{|\ga|< N} \frac{b^\ga}{\ga!} \partial^\ga_\xi p(\xi;\mu) +r_N(\xi;\mu).
\end{eqnarray*}
Then $r_N\in S^{\fm-N,d}(\R^k; S)$.
\end{lemma} 

\begin{proof}
(a) and (b) are immediate consequences of the facts that there exist constants $c,C>0$ such that 
\begin{eqnarray}%
c\skp{(\xi,\eta)}&\le &\skp{(\xi,\eta,\eta)}\le C\skp{(\xi,\eta)}\text{ and}\nonumber\\
c\skp{\xi}&\le &\skp{B\xi+b}\le C\skp{\xi}.\label{4.5.2}
\end{eqnarray} 

(c) We have 
\begin{eqnarray*}%
r_N(\xi,\mu) =N  \sum_{|\beta|=N}\frac{b^\gb}{\gb!}\int_0^1 
(1-\gt)^{N-1} \partial^\gb_\xi p(\xi+\gt b;\mu)\, d\gt.
\end{eqnarray*}
The inequalities \eqref{4.5.2} imply that $\partial^\gb_\xi p(\xi+\gt b;\mu)\in
S^{m-N,d}$, uniformly in $\gt$, hence this is also true for $r_N$.  
\end{proof}


\end{document}